\documentclass[10pt]{article}
\usepackage[english]{babel} %togliere italian per avere proof invece di dimostrazione
\usepackage[T1]{fontenc}
\usepackage[utf8]{inputenc}
\usepackage[autostyle,italian=guillemets]{csquotes}
\usepackage[title]{appendix}
\usepackage{authblk}

\usepackage{geometry}
\usepackage{color}
\usepackage[svgnames, x11names]{xcolor}
 \definecolor{red}{rgb}{1,0,0}
 \definecolor{darkblue}{rgb}{0,0,0.7}
 \definecolor{white}{rgb}{1,1,1}
 \definecolor{ocra}{rgb}{1, 0.7, 0}
 \definecolor{pink}{rgb}{1, 0, 1}
 \definecolor{orange}{rgb}{1, 0.5, 0}
 \definecolor{purple}{cmyk}{0.4, 1, 0, 0,}
\usepackage{comment}
\usepackage{mathtools, amssymb, mathrsfs, yfonts, amsbsy, amsmath, amsfonts, braket, amsthm, faktor, booktabs, wasysym, stmaryrd, euscript, accents, amscd, array, tabularx, longtable, adjustbox}
\usepackage{accents}
\newcommand{\ubar}[1]{\underaccent{\bar}{#1}}
\usepackage[shortlabels]{enumitem}

\usepackage{moresize, scalerel}
\usepackage[nottoc]{tocbibind} %bibliografia con bibtex
\usepackage[pdfencoding=auto, psdextra]{hyperref}
\hypersetup{colorlinks,breaklinks,
            linkcolor=blue,urlcolor=blue,
            anchorcolor=blue,citecolor=blue}
\usepackage{url, todonotes, soul}

\usepackage{tikz, tikz-cd, graphicx, tikzducks}
\usetikzlibrary{arrows,shapes,automata,backgrounds,petri,positioning, graphs}

\usepackage{fancyhdr}

\usepackage{booktabs}
\usepackage{float}
\restylefloat{table}

\newtheorem{theoremA}{Theorem}

\theoremstyle{plain}
\newtheorem{theorem}{Theorem}[section]
\newtheorem{corollary}{Corollary}[theorem]
\newtheorem{lemma}[theorem]{Lemma} 
\newtheorem{prop}[theorem]{Proposition}
\theoremstyle{definition}
\newtheorem{definition}[theorem]{Definition}
\theoremstyle{remark}
\newtheorem{remark}[theorem]{Remark} %usare newtheorem* per non numerarlo 
\theoremstyle{remark}
\newtheorem{example}[theorem]{Example}

\numberwithin{equation}{section}

\newcommand{\mysetminusD}{\hbox{\tikz{\draw[line width=0.6pt,line cap=round] (3pt,0) -- (0,6pt);}}}
\newcommand{\mysetminusT}{\mysetminusD}
\newcommand{\mysetminusS}{\hbox{\tikz{\draw[line width=0.45pt,line cap=round] (2pt,0) -- (0,4pt);}}}
\newcommand{\mysetminusSS}{\hbox{\tikz{\draw[line width=0.4pt,line cap=round] (1.5pt,0) -- (0,3pt);}}}
\newcommand{\mysetminus}{\mathbin{\mathchoice{\mysetminusD}{\mysetminusT}{\mysetminusS}{\mysetminusSS}}}

\newcommand{\N}{\mathbb{N}}
\newcommand{\Z}{\mathbb{Z}}

\newcommand{\C}{\mathbb{C}}

\newcommand{\K}{\mathbb{K}}
\newcommand{\I}{\mathcal{I}}
\newcommand{\J}{\mathcal{J}}

\renewcommand{\P}{\mathbb{P}}

\newcommand{\addots}{\text{\reflectbox{$\ddots$}}}
\renewcommand{\tilde}[1]{\widetilde{#1}}

\DeclareMathOperator{\Span}{Span}

\DeclareMathOperator{\dimv}{\underline{dim}}

\DeclareMathOperator{\Gr}{Gr}

\makeatletter
\newcommand{\address}[1]{\gdef\@address{#1}}
\newcommand{\email}[1]{\gdef\@email{\url{#1}}}
\newcommand{\@endstuff}{\par\vspace{\baselineskip}\noindent\small
\begin{tabular}{@{}l}\scshape\@address\\\textit{E-mail address:} \@email\end{tabular}}
\AtEndDocument{\@endstuff}
\makeatother

\title{Some Lagrangian quiver Grassmannians for \\ the equioriented cycle}
\author{Matteo Micheli}
\address{Dipartimento di Matematica "Guido Castelnuovo"\\ Sapienza Universit\`a di Roma, Piazzale Aldo Moro 5, 00185 Rome, Italy}
 \email{matteomicheli98@outlook.com}
\date{}

\begin{document}

\maketitle

\begin{abstract}
The goal of this paper is to better understand a family of linear degenerations of the classical Lagrangian Grassmannians $\Lambda(2n)$. It is the special case for $k=n$ of the varieties $X(k,2n)^{sp}$, introduced in previous joint work with Evgeny Feigin, Martina Lanini and Alexander Pütz. These varieties are obtained as isotropic subvarieties of a family of quiver Grassmannians $X(n,2n)$, and are acted on by a linear degeneration of the algebraic group $Sp_{2n}$. We prove a conjecture proposed in the paper above for this particular case, which states that the ordering on the set of orbits in $X(n,2n)^{sp}$ given by closure-inclusion coincides with a combinatorially defined order on what are called \emph{symplectic} $(n,2n)$-\emph{juggling patterns}, much in the same way that the $Sp_{2n}$ orbits in $\Lambda(2n)$ are parametrized by a type C Weyl group with the Bruhat order. The dimension of such orbits is computed via the combinatorics of \emph{bounded affine permutations}, and it coincides with the length of some permutation in a Coxeter group of affine type C. Furthermore, the varieties $X(n,2n)$ are GKM, that is, they have trivial cohomology in odd degree and are equipped with the action of an algebraic torus with finitely many fixed points and 1-dimensional orbits. In this paper it is proved that $X(n,2n)^{sp}$ is also GKM, with respect to the action of a subtorus of the above torus.
\end{abstract}
\newpage

\section{Introduction}

The linear degeneration $X(k,N)$ of the classical Grassmannian $\Gr(k,N)$ was studied in \cite{ML1}. It arises as a quiver Grassmannian for a nilpotent representation $U_{[N]}$ of the equioriented cycle on $N$ vertices, and enjoys several nice properties analogous to those of the classical Grassmannian. The orbits for the automorphism group $G$ of $U_{[N]}$ provide a cellular decomposition of $X(k,N)$, and the poset of orbits, with the order relation given by closure inclusion, has an explicit combinatorial description in terms of \emph{juggling patterns}. This poset is also isomorphic to a lower order ideal of a Coxeter group of affine type A with the Bruhat order, and the dimension of an orbit is computed as the length of the corresponding permutation in the Coxeter group. This order is recorded in the 1-skeleton for the action of an algebraic torus $T$ of dimension $N+1$ on $X(k,N)$, and the $T$-equivariant cohomology is easy to compute since $X(k,N)$ is a GKM variety \cite{GKM, LaniniPuetz}.

In \cite{ours}, the subvariety $X(k,2n)^{sp}$ of $X(k,2n)$ was defined, by introducing a notion of "orthogonal subrepresentation" with respect to a suitable non degenerate, skew-symmetric bilinear form. It is a linear degeneration of the classical isotropic Grassmannian $\Gr(k,2n)^{sp}$, and was proved to inherit many interesting properties from the ambient variety. The cells of $X(k,2n)$, when intersected with the subvariety, produce once again a cellular decomposition, whose cells are orbits for the action of the subgroup $G^{sp}$ of $G$ consisting of elements that preserve the symplectic form. The notion of \emph{symplectic juggling patterns} was defined, and these objects parametrize the $G^{sp}$-orbits, called symplectic cells. In the same paper it was conjectured that the combinatorial order inherited by the symplectic juggling patterns, by being a subset of the set of juggling patterns, corresponds once again to the $G^{sp}$-orbit closure inclusion order.

In this paper we focus on the degeneration $X(n,2n)^{sp}$ of the Lagrangian Grassmannian $\Lambda(2n)$, i.e. the case where $k$, the dimension of the vector subspaces of $\C^{2n}$, is exactly half the dimension of the ambient vector space. This implies that the subvariety consists of fixed points for an automorphism of the ambient variety, therefore we can transfer this symmetry to combinatorial and algebraic structures. The symmetry on the poset of juggling patterns extends to an automorphism of a Coxeter group of affine type $A_{2n-1}^{(1)}$, realized as affine permutations, thus its fixed-point subgroup is a Coxeter group of affine type $C_n^{(1)}$ \cite{karpman}. The intersection of this subgroup with the lower order ideal of elements corresponding to juggling patterns is the poset that parametrizes the $G^{sp}$-orbits, and the Bruhat order on this set corresponds to the geometric closure inclusion order on the set of orbits. To prove this result, we define the action of an $(n+1)$-dimensional algebraic torus $T^{sp}$ on $X(n,2n)^{sp}$, and show that the relations for the geometric order are encoded by its 1-dimensional orbits, which correspond to Bruhat order relations in the type $C$ group given by one reflection. The main results are the following:
\begin{theoremA}[Theorem \ref{thm:mainresult1}]
    The $T^{sp}$-action makes $X(n,2n)^{sp}$ a GKM variety.
\end{theoremA}

\begin{theoremA}[Theorem \ref{thm:mainresult2}]
    The closure of a symplectic cell in $X(n,2n)^{sp}$ is the intersection of $X(n,2n)^{sp}$ with the closure of the corresponding affine cell in $X(n,2n)$.
\end{theoremA}

\begin{theoremA}[Theorem \ref{thm:mainresult3}]
    For any symplectic $(n,2n)$-juggling pattern $\J$, the dimension of $C_\J^{sp}$ is equal to the length of the corresponding bounded affine permutation $f_\J$ in $C^0_{n+1}$.
\end{theoremA}

Theorem B answers Conjecture 4.9 from \cite{ours}, for the Lagrangian case. The paper is structured as follows: in Section \ref{sec:background} we set some notation, recall definitions of quiver representation theory, and introduce the main objects. Here the varieties $X(k,N)$ and $X(n,2n)^{sp}$ are defined, together with juggling patterns and bounded affine permutations. Section \ref{sec:combinatorics} contains most of the combinatorics needed for later considerations, and the symmetry from the variety is transferred to the aforementioned posets. In Sections \ref{sec:torusaction1} and \ref{sec:torusaction2} we describe the action of an algebraic torus $T$ on $X(k,2n)$ and define a subtorus $T^{sp}$ which acts on $X(n,2n)^{sp}$ in a skeletal way. Here we study its 0- and 1- dimensional orbits, try to understand what information they contain, and prove Theorem A. Finally, Section \ref{sec:mainresults} will combine the tools developed up until then into Theorems B and C. As a consequence of Theorem A, in Appendix \ref{appendix} we are able to compute the $T^{sp}$-equivariant cohomology of $X(2,4)^{sp}$, the smallest nontrivial example.

\smallskip

\subsubsection*{Acknowledgements} The author is a member of the INdAM-G.N.S.A.G.A network. He also thanks his PhD supervisor, Martina Lanini, for her continued guidance, as well as Evgeny Feigin and Alexander P\"utz  for many insightful conversations.

\section{Background} \label{sec:background}
\subsection{Quiver Grassmannians}
We first recall some useful notions. They will be needed to introduce the main objects of study, which are realized as subvarieties of some quiver Grassmannians.

\begin{definition}
    A \emph{quiver} $Q = (Q_0,Q_1,s,t)$ is an oriented graph, with vertex set $Q_0$ and edge set $Q_1$, equipped with maps $s,t \colon Q_1 \longrightarrow Q_0$ which give the orientation of each edge by specifying its source and target vertices respectively. The oriented edges are called \emph{arrows}. The notation $\alpha \colon i \longrightarrow j$ represents an arrow with source $i$ and target $j$.
\end{definition}

\begin{definition}
Given a quiver $Q$ and a field $\K$, a \emph{representation} of $Q$ over $\K$ is a collection $M=(M_i,M_\alpha)$ of finite dimensional $\K$-vector spaces $(M_i)_{i \in Q_0}$ over the vertices and of $\K$-linear maps $(M_\alpha \colon M_{s(\alpha)} \longrightarrow M_{t(\alpha)})_{\alpha \in Q_1}$ over the arrows. A \emph{morphism} $\varphi \colon M \longrightarrow M'$ of $Q$-representations over $\K$ is a collection $(\varphi_i)_{i \in Q_0}$ of $\K$-linear maps $\varphi_i \colon M_i \longrightarrow M'_i$ between the vector spaces of $M$ and $M'$ over each vertex, such that for any arrow $\alpha \colon i \longrightarrow j \in Q_1$ the following square
\[\begin{tikzcd}
    M_i \arrow[r,"\varphi_i"] \arrow[d, "M_\alpha"'] & M'_i \arrow[d, "M'_\alpha"] \\ M_j \arrow[r, "\varphi_j"] & M'_j
\end{tikzcd}\]
is commutative. A \emph{subrepresentation} $U \subseteq M$ is a collection of vector subspaces $(U_i \subseteq M_i)_{i \in Q_0}$ such that, for any arrow $\alpha \colon i \longrightarrow j$, the inclusion $M_\alpha (U_i) \subseteq U_j$ is satisfied. The collection of dimensions of the vector spaces $M_i$ forms an element of $\N^{Q_0}$ called the \emph{dimension vector} of $M$, which is denoted by $\dimv M$.
\end{definition}

\begin{definition}
    Let $\ubar{d} \in \N^{Q_0}$ and let $M$ be a representation of $Q$. The space
    \[\Gr_{\ubar{d}}(M) \coloneqq \{U \subseteq M \, \vert \, \dimv U = \ubar{d}\}\]
    is called the \emph{quiver Grassmannian} of subrepresentations of $M$ with dimension vector $\ubar{d}$.
\end{definition}

\begin{remark}
    Quiver Grassmannians are projective varieties, since the product of classical Grassmannians embeds in a projective space, and the conditions $M_\alpha(U_i) \subseteq U_j$ are closed.
\end{remark}

\subsection{The Juggling variety}

From now on we will work with the field of complex numbers $\K =\C$. First we fix two positive integers $k \le N$. We will write $[N]$ for the set $\Set{1, 2, \dots, N}$, $[k,N]$ for $\Set{k, k+1, \dots, N}$ and $\binom{[N]}{k}$ for the set of cardinality $k$ subsets of $[N]$. Our quiver of choice is the equioriented cycle on $N$ vertices, denoted by $\Delta_N$, with vertex set indexed by $\Z_N$, the integers modulo $N$, in such a way that for all $i \in \Z_N$ there is an arrow $\alpha_i \colon i \longrightarrow i+1$.

\begin{center}
    \begin{tikzpicture}
     \begin{scope}[every node/.style={circle, draw=black!100,
     very thin,minimum size=1mm}]
    \node (0) at (90:3cm) {0};
    \node (1) at (45:3cm) {1};
    \node (2) at (0:3cm) {2};
    \node (3) at (-45:3cm) {3};
    \node (4) at (-90:3cm) {4};
    \node (5) at (225:3cm) {5};
    \node (6) at (180:3cm) {6};
    \node (7) at (135:3cm) {7};
\end{scope}

\begin{scope}[every edge/.style= 
              {draw=black,thick}]
\path [->] (0) edge[bend left=19, "$\alpha_0$"] (1);
\path [->] (1) edge[bend left=19, "$\alpha_1$"] (2);
\path [->] (2) edge[bend left=19, "$\alpha_2$"] (3);
\path [->] (3) edge[bend left=19, "$\alpha_3$"] (4);
\path [->] (4) edge[bend left=19, "$\alpha_4$"] (5);
\path [->] (5) edge[bend left=19, "$\alpha_5$"] (6);
\path [->] (6) edge[bend left=19, "$\alpha_6$"] (7);
\path [->] (7) edge[bend left=19, "$\alpha_7$"] (0);
\end{scope}
    \end{tikzpicture}
    
    The quiver $\Delta_8$.
    \end{center}
Our chosen $\Delta_N$-representation, denoted by $U_{[N]}$, has $\C^N$ on every vertex and the linear map $s$ on every arrow, defined as follows:
\begin{equation}\label{eq:s}
s(e_j) = \begin{cases} e_{j+1} & j \ne N \\ 0 & j=N \end{cases}
\end{equation}
where $e_1, \dots, e_N$ denotes the standard basis of $\C^N$. We denote with $G$ the automorphism group of $U_{[N]}$; here $N$ is fixed, so we dropped the subscript. An element of $G$ is a collection of invertible $N \times N$ matrices $(g_i)_{i \in \Z_N}$ such that $s \cdot g_i = g_{i+1} \cdot s$ for all $i$.

Now let $\ubar{k}$ be the dimension vector constant to $k$ for $\Delta_N$. The quiver Grassmannian that we will consider is the following:
\[X(k,N) = \text{Gr}_{\ubar{k}}\left(U_{[N]}\right) .\]
A point of $X(k,N)$ consists therefore of a collection $(V_i)_{i \in \Z_N}$ of $k$-dimensional subspaces of $\C^N$, such that for every vertex $i$ the inclusion $s(V_i) \subseteq V_{i+1}$ holds. The group $G$ acts on the quiver Grassmannian and partitions it into orbits, each containing exactly one point consisting of subspaces spanned by subsets of the standard basis, i.e. coordinate subspaces \cite[Theorem 4.10]{ML1}. Thus the orbits are parametrized by combinatorial objects known as \emph{juggling patterns}, see \cite[Section 3]{KnutsonLamSpeyer}.

\begin{definition}
    A $(k,N)$\emph{-juggling pattern} is a collection $\J = (J_i)_{i \in \Z_N}$ of cardinality $k$ subsets of $[N]$, such that for any $i \in \Z_N$ and $j \in J_i \mysetminus \{N\}$, we have $j+1 \in J_{i+1}$. The set of $(k,N)$-juggling patterns will be denoted by $JP(k,N)$.
\end{definition}

\begin{remark}
    Another $\Delta_N$-representation and its corresponding quiver Grassmannian are studied in \cite{ML1}. This other realization produces the combinatorial set of $(k,N)$-\emph{Grassmann necklaces}. In that paper, the authors remark that one realization is more suited than the other depending on what one wants to prove. The sets of $(k,N)$-juggling patterns and $(k,N)$-Grassmann necklaces are in bijection, since there is an isomorphism between the two representations (and thus the two quiver Grassmannians) which sends coordinate points to coordinate points.
\end{remark}

Given a subset $I$ of $[N]$, we define $V_I$ as the subspace of $\C^N$ spanned by the standard basis vectors with indices in $I$. Then for any $(k,N)$-juggling pattern $\J$ there is a corresponding point $p_\J$ of $X(k,N)$, defined by $(p_\J)_i \coloneqq V_{J_i}$, which we will call a \emph{juggling pattern point}. The $G$-orbit of $p_\J$ is an affine cell \cite[Theorems 1-2]{ML1} and will be denoted by $C_\J$. The affine coordinates for $C_\J$ are given by tuples
\[\left(u^{(a)}_{i,j} \, \big\vert \, \, a \in \Z_N, \, \, j \in J_a, \, \, J_a \not \ni i > j \right)\]
that are subject to conditions of the form $u^{(a)}_{i,j} = u^{(a+1)}_{i+1,j+1}$ \cite[Theorem 5.7]{LaniniPuetz}. That is, they are parametrized by equivalence classes of triples $(a,i,j)$ under the relation $(a,i,j) \sim (a+1,i+1,j+1)$. If a point $V \in C_\J$ has such coordinates, a basis for the vector space $V_a$ on the vertex $a$ is given by
\begin{equation}\label{eq:coordinatebasis}
\left \{v^{(a)}_j \coloneqq e^{(a)}_j + \sum_{\substack{i>j \\ i \notin J_a}} u^{(a)}_{i,j} e^{(a)}_i \right\}_{j \in J_a}\end{equation}
where $e_j^{(a)}$ is the $j$-th standard basis element for $U_{[N]}^{(a)}$, the copy of $\C^N$ on the vertex $a$. Every orbit for the action of $G$ is of the form $C_\J$ for some $\J$, and since $G$ is an algebraic group, their set can be given a partial order defined by the inclusion of their closures:
\[C_\J \le C_{\J'} \overset{\text{def}}{\iff} C_\J \subseteq \overline{C_{\J'}} \, \, .\]
This order can thus be transferred to the set of juggling patterns by the above correspondence:
\[\J \le \J' \overset{\text{def}}{\iff} p_\J \in \overline{C_{\J'}} \, .\]
The closure inclusion order can also be expressed combinatorially. For two subsets \[A = \{a_1, a_2, \dots, a_k\} \, \, \text{and} \, \, B=\{b_1, b_2, \dots, b_k\}\]of $[N]$, ordered increasingly, we say that
\begin{equation}\label{eq:orderforsets}
	A \le B \overset{\text{def}}{\iff}  a_i \le b_i \, \, \forall \, \, i \in [k] \, .
\end{equation}
This allows us to recall the next result \cite[Proposition 7.3]{ML1}:
\begin{prop}
    For two juggling patterns $\J$ and $\J'$, the condition $\J \le \J'$ is equivalent to $J_i \ge J'_i$ for all $i \in \Z_N$.
\end{prop}

\subsection{The symplectic subvariety}

In order to endow the vector spaces of our representations with symplectic forms, we focus on the case of cycles with an even number of vertices, i.e. when $N = 2n$. This provides both $\Delta_{2n}$ and the representation $U_{[2n]}$ with a particular symmetry.

\begin{definition}
    A \emph{symmetric quiver} is a tuple $(Q, \tau, \varsigma)$, where $Q$ is a quiver, $\tau$ is an involution on both the set of vertices $Q_0$ and the set of edges $Q_1$, such that
    \begin{itemize}
        \item $s \tau(\alpha) = \tau(t \alpha)$ for all $\alpha \in Q_1$;
        \item $t \tau(\alpha) = \tau(s \alpha)$ for all $\alpha \in Q_1$;
        \item if $s \alpha = t \alpha$ then $\tau(\alpha) = \alpha$,
    \end{itemize}
    and $\varsigma$ is a map $Q_0^\tau \cup Q_1^\tau \rightarrow \{\pm 1\}$ which attaches a sign to any fixed vertex or edge. We make a choice, and partition $Q_0$ into $Q_0^\tau \sqcup Q_0^+ \sqcup Q_0^-$, where $\tau(Q_0^+)= Q_0^-$, and do the same for $Q_1$. A dimension vector $\ubar{d} \in \N^{Q_0}$ is \emph{symmetric} if $d_i = d_{\tau(i)}$ for all $i \in Q_0$. A \emph{symmetric representation} $V$ for $(Q, \tau, \varsigma)$ is a representation such that
    \begin{enumerate}[I.]
    \item $\dimv V$ is symmetric;
    \item $V_i \cong V_{\tau(i)}^*$ for all $i \in Q_0^-$;
    \item for $i \in Q_0^\tau$, $V_i$ is equipped with a scalar product if $\varsigma(i) = 1$ or a (non-degenerate) symplectic form if $\varsigma(i) = -1$. Observe that either form provides a natural isomorphism $\psi_i \colon V_i \rightarrow V_i^*$;
    \item If $\alpha \in Q_1^\tau$, then $V_\alpha = \varsigma(\alpha) \cdot V_\alpha^*$, where $V^{**}$ is identified with $V$ canonically;
    \item If $\alpha \in Q_1^+$ and neither its source nor target are $\tau$-fixed, then $V_{\tau{\alpha}} = V_\alpha^*$;
    \item If $\alpha \in Q_1^+$ with $t \alpha \in Q_0^\tau$, then $V_{\tau(\alpha)} = V_\alpha^* \circ \psi_{t \alpha}$;
    \item If $\alpha \in Q_1^+$ with $s \alpha \in Q_0^\tau$, then $V_{\tau(\alpha)} = \psi_{s\alpha} \circ V_\alpha^*$.
\end{enumerate}
We work in the field of complex numbers, but this definition is valid in any characteristic other than 2.
\end{definition}

\begin{lemma}\label{lem:dual}
     Let $n$ be a positive integer and $\Omega$ an invertible $2n \times 2n$ matrix such that
     \begin{equation}\label{eq:omega}
    \Omega^t = \Omega^{-1} = - \Omega \, \, ,\end{equation}
     so that it is the Gram matrix for a symplectic form $(-,-)_\Omega$ on $\C^{2n}$. Then for every matrix $A \in \text{M}_{2n \times 2n}(\C)$ and every pair of vectors $v, w \in \C^{2n}$ we have
    \[(Av,w)_\Omega = (v, -\Omega A^t \Omega w)_\Omega \, .\]
    Specifically, given any subspace $V$, we have
    \[\Omega A^t \Omega (AV)^\perp \subseteq V^\perp \, ,\]
    where $V^\perp$ is the subspace orthogonal to $V$ with respect to $\Omega$. Moreover, if $A$ is invertible then the equality holds.
\end{lemma}

\begin{proof}
    The following chain of equalities holds:
    \[(Av,w)_\Omega = (Av)^t \Omega w = v^t A^t \Omega w = v^t (-\Omega^2) A^t \Omega w = - v^t \Omega^2 A^t \Omega w = -(v, \Omega A^t \Omega w)_\Omega \, \, .\]
\end{proof}

Consider now the map $\tau$ on $\Delta_{2n}$ that swaps vertices $i$ and $-i$. It is clearly a quiver symmetry, whose fixed vertices are $0$ and $n$; it swaps the edges underlying the arrows $\alpha_i$ and $\alpha_{-i-1}$, therefore fixing none. The choice of $\varsigma$ is arbitrary, so we pick $\varsigma(0) = \varsigma(n) = -1$ since we want to work with symplectic forms. We set $Q_0^+ = \{1, \dots, n-1\}$ and $Q_1^+ = \{\alpha_0, \dots, \alpha_{n-1}\}$. Below, the symmetry for $n=4$ is shown, with $Q_0^+$ and $Q_1^+$ in red, and their images in blue.

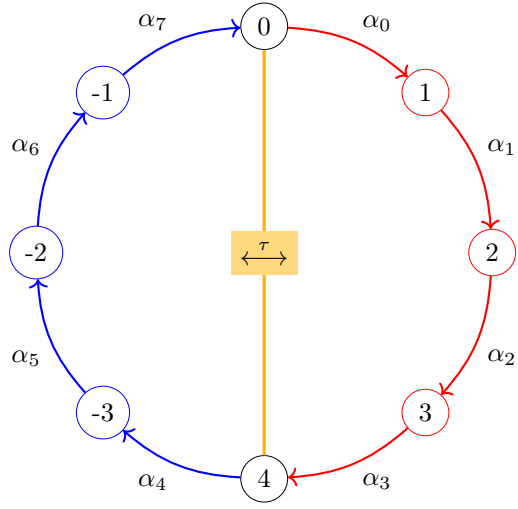
\begin{center}
    \begin{tikzpicture}
     \begin{scope}[every node/.style={circle,
     very thin,minimum size=1mm}]
    \node[draw=black!100] (0) at (90:3cm) {0};
    \node[draw=red!100] (1) at (45:3cm) {1};
    \node[draw=red!100] (2) at (0:3cm) {2};
    \node[draw=red!100] (3) at (-45:3cm) {3};
    \node[draw=black!100] (4) at (-90:3cm) {4};
    \node[draw=blue!100] (5) at (225:3cm) {-3};
    \node[draw=blue!100] (6) at (180:3cm) {-2};
    \node[draw=blue!100] (7) at (135:3cm) {-1};
\end{scope}

\begin{scope}[every edge/.style= 
              {thick}]
\path [->] (0) edge[bend left=19, "$\alpha_0$",draw=red] (1);
\path [->] (1) edge[bend left=19, "$\alpha_1$",draw=red] (2);
\path [->] (2) edge[bend left=19, "$\alpha_2$",draw=red] (3);
\path [->] (3) edge[bend left=19, "$\alpha_3$",draw=red] (4);
\path [->] (4) edge[bend left=19, "$\alpha_4$",draw=blue] (5);
\path [->] (5) edge[bend left=19, "$\alpha_5$",draw=blue] (6);
\path [->] (6) edge[bend left=19, "$\alpha_6$",draw=blue] (7);
\path [->] (7) edge[bend left=19, "$\alpha_7$",draw=blue] (0);
\draw[-] (0) edge[draw=ocra,very thick] node[fill=ocra!50] {$\overset{\tau}{\longleftrightarrow}$} (4);
\end{scope}
    \end{tikzpicture}
    \end{center}

\begin{prop}
    The representation $U_{[2n]}$, for $(\Delta_{2n}, \tau, \varsigma)$, is symmetric.
\end{prop}

\begin{proof}
We fix a skew-symmetric $2n \times 2n$ matrix which will be the Gram matrix for our chosen symplectic form on $\C^{2n}$ in the standard basis:
\[\Omega \coloneqq \begin{pmatrix}
    0 & 0 & \cdots & 0 & 1 \\
    0 & 0 & \cdots & -1 & 0 \\
    \vdots & \vdots & \addots & \vdots & \vdots \\
    0 & 1 & \cdots & 0 & 0 \\
    -1 & 0 & \cdots & 0 & 0
\end{pmatrix}_{\, .}\]
We choose this specific form because it satisfies
\begin{equation}\label{eq:goodomega}
    s^t = \Omega \cdot s \cdot \Omega \, \, .
\end{equation}
The first condition for symmetry is trivially satisfied since $\dimv U_{[2n]}$ is constant to $2n$. Equipping each vector space with the symplectic form $\Omega$ allows us to meet the third condition and to canonically identify $\C^{2n}$ with its dual, in particular on the vertices in $Q_0^-$ so the second condition also holds. The rest are fulfilled because of Lemma \ref{lem:dual} and \eqref{eq:goodomega}.
\end{proof}

The following Lemma is \cite[Proposition 2.4]{ours}.

\begin{lemma}\label{lem:tau}
The image of the map
\begin{align*}
    \tau \colon X(k,2n) &\longrightarrow \prod_{i \in \Z_{2n}} {\text{Gr} (2n-k,2n)} \\
(V_i)_i \, \, & \longmapsto \quad (V_{-i}^\perp)_i
\end{align*}
is $X(2n-k,2n)$, and
\[ X(k,2n) \overset{\tau}{\longrightarrow} X(2n-k,2n) \overset{\tau}{\longrightarrow} X(k,2n)\]
is the identity map. In particular, when $k = n$, $\tau$ is an automorphism of $X(n, 2n)$ of order 2.
\end{lemma}

\begin{definition}
    We define $X(n,2n)^{sp} \coloneqq X(n,2n)^\tau$, the locus of Langrangian subrepresentations, i.e. $\tau$-fixed points. They will be called \emph{symplectic points}.
\end{definition}

In \cite{ours} the authors defined the subvariety $X(k,2n)^{sp}$ of \emph{isotropic} points for any $k \in [2n]$, but we focus on the case $k=n$ because the automorphism of the variety provides additional insight. Now we recall the morphism of algebraic groups that will allow us to find a linear degeneration of $Sp_{2n}$ which acts on $X(n,2n)^{sp}$. We define
\begin{align*}
\sigma \colon G \longrightarrow & \prod_{i \in \Z_{2n}} GL_{2n} \\
(g_i)_i \longmapsto & (-\Omega \cdot g_{-i}^{-t} \cdot \Omega)_i
\end{align*}
where $g^{-t}$ denotes $(g^{-1})^t$, and all matrices represent linear maps with respect to the standard basis.

\begin{lemma}\label{lem:sigma}
    The image of $\sigma$ is $G$. It also satisfies
    \begin{equation}\label{eq:sigmatau} \tau g \tau (x) = \sigma(g)(x)
    \end{equation}
    for all $g \in G$ and $x \in X(n,2n)$
\end{lemma}

\begin{proof}
Remember that an element $(g_i)_i$ of $G$ satisfies $s \cdot g_i = g_{i+1} \cdot s$ for all vertices $i$, as does its inverse. Therefore, for the first statement to hold, we need to prove that $s \cdot \sigma(g)_i = \sigma(g)_{i+1} \cdot s$ for all $i \in \Z_{2n}$. Taking the transpose of both sides in \eqref{eq:goodomega} we obtain $\Omega \cdot s^t \cdot \Omega = s$, thus
\begin{gather*}    
\sigma(g)_{i+1} \cdot s = -\Omega \cdot g_{-i-1}^{-t} \cdot \Omega \cdot s = \Omega \cdot g_{-i-1}^{-t} \cdot s^t \cdot \Omega = \Omega \cdot \bigl(s \cdot (g_{-i-1}^{-1}) \bigr)^t \cdot \Omega = \\
= \Omega \cdot \bigl( (g_{-i}^{-1}) \cdot s \bigr)^t \cdot \Omega = \Omega \cdot s^t \cdot g_{-i}^{-t} \cdot \Omega = -\Omega \cdot s^t \cdot \Omega^2 \cdot g_{-i}^{-t} \cdot \Omega = -s \cdot \Omega \cdot g_{-i}^{-t} \cdot \Omega = s \cdot \sigma(g)_i \, .
\end{gather*}
Lastly, we see that $\tau$ and $\sigma$ satisfy $\eqref{eq:sigmatau}$ by Lemma \ref{lem:dual} and by the definition of $\tau$. This is also \cite[Lemma 3.2]{ours}.
\end{proof}

Let $G^{sp}$ be the subgroup of symplectic elements of $G$, i.e. those $g=(g_i)_i$ which satisfy $(g_iv,g_{-i}w)=(v,w)$ for all $i \in \Z_{2n}$, $v \in U^{(i)}_{[2n]}$ and $w \in U^{(-i)}_{[2n]}$, and this subgroup coincides with $G^\sigma$. We are interested in studying the decomposition of $X(n,2n)^{sp}$ into strata, which are the $G^{sp}$-orbits. The next result was proved for generic $k$ as  \cite[Proposition 3.9]{ours}, but in the Lagrangian case one could also apply \cite[Section 2.1]{magyarwz} for a more steamlined proof.

\begin{prop}\label{prop:orbitintersection}
    The following equality holds for all $x \in X(n,2n)^{sp}$:
\begin{equation}\label{eq:goodorbits1}
G^{sp} \cdot x = (G \cdot x) \cap X(n,2n)^{sp} \, \, .
\end{equation}
In addition, the $G^{sp}$-orbits are affine spaces.
\end{prop}

\section{Combinatorics}\label{sec:combinatorics}
Now we are going to transfer the symmetry from $X(n,2n)$ to the poset of $(n,2n)$-juggling patterns. Observe that $\Omega$ pairs the basis vector $e_i$ with $e_{2n-i+1}$. We write $\tilde{i}$ in place of $2n-i+1$ for brevity. In particular $(e_i, e_j)_\Omega = (-1)^{i+1} \delta_{j, \tilde{i}}$. Given a subset $J$ of $[2n]$, we define
\[R(J) \coloneqq [2n] \mysetminus \{\tilde{j} \, \vert \, j \in J\} \, \, .\]
Clearly $R \bigl(R(J) \bigr)=J$ and $V_J^\perp = V_{R(J)}$. Here $V_J$ is the coordinate subspace of $\C^{2n}$ spanned by the standard basis elements with indices in $J$. Also $\vert R(J) \vert = 2n - \vert J \vert$. In particular it is an involution on the set of cardinality $n$ subsets of $[2n]$. We extend $R$ to the set of $(n,2n)$-juggling patterns, sending $\mathcal{J} = (J_i)_i$ to $R(\mathcal{J}) \coloneqq \bigl( R(J_{-i}) \bigr)_i$. This is another $(n,2n)$-juggling pattern: first, observe that the restriction of the linear map $s$ on the standard basis produces the following function
\begin{align*}
[2n-1] \longrightarrow& \, \, [2,2n] \\
i \quad \longmapsto& \, \, i+1 \, .
\end{align*}
which will be again called $s$, with a slight notation abuse. Now observe that $R$ satisfies
\begin{gather*}
    R(I \cup J) = R(I) \cap R(J) \\
    R(I \cap J) = R(I) \cup R(J) \\
    I\subseteq J \Longrightarrow R(J) \subseteq R(I) \, \, .
\end{gather*}
Moreover, given $J \in \begin{pmatrix}
    [2n] \\ n
\end{pmatrix}$ that does not contain $2n$, we have that $1 \in R(J)$ and that $1 \notin sJ$, therefore $2n \in R(sJ)$. Then
\[
R(sJ) = [2n] \mysetminus \{\tilde{j+1} \, \vert \, j \in J\} = [2n] \mysetminus \{\tilde{j}-1 \, \, \vert \, j \in J\} = \{a \in [2n] \, \vert \, a+1 \ne \tilde{j} \, \, \forall \, j \in J\}
\]
thus
\begin{equation}\label{eq:srs}
s\bigl( R(sJ) \mysetminus \{2n\} \bigr) = \{a+1 \, \vert \, a \in R(sJ) \, \land \, a \ne 2n\} = R(J) \mysetminus \{1\} \, \, .\end{equation}
Now considering a juggling pattern $\J$, we want to show that $s\bigl( (R\J)_i \mysetminus \{2n\} \bigr) \subseteq (R\J)_{i+1}$.
By hypothesis $s(J_{-i-1} \mysetminus \{2n\}) \subseteq J_{-i}$, and so
\[R(\J)_i = R(J_{-i}) \subseteq R\bigl( s(J_{-i-1} \mysetminus \{2n\}) \bigr) \]
which implies
\begin{gather*}
    s \bigl( (R\J)_i \mysetminus \{2n\} \bigr) = s \bigl( R(J_{-i}) \mysetminus \{2n\} \bigr) \subseteq s \Bigl( R \bigl(s(J_{-i-1} \mysetminus \{2n\}) \bigr) \mysetminus \{2n\} \Bigr) \overset{\eqref{eq:srs}}{=} \\ \overset{\eqref{eq:srs}}{=} R(J_{-i-1} \mysetminus \{2n\}) \mysetminus \{1\} \overset{\triangle}{\subseteq} R(J_{-i-1}) = (R\J)_{i+1}.
\end{gather*}
The inclusion {\scriptsize $\triangle$} is an equality when $2n \in J_{-i-1}$ and is a strict inclusion otherwise.

\begin{remark}
    Notice that $R\bigl(R(\J) \bigr) = \J$ for any juggling pattern $\J$. Although we are focusing on the Lagrangian case, the above computation also works for $(k,2n)$-juggling patterns with $k \in [n]$, and all results and definitions up until Lemma \ref{lem:length} included can be extended to the other isotropic cases. 
\end{remark}

\begin{example}\label{ex:running1}
    The $(2,4)$-juggling pattern \[\J = \begin{matrix}
        &\{2,4\}& \\
        \{3,4\}&&\{3,4\} \\
        &\{3,4\}&
    \end{matrix}\]
    satisfies $\J=R\J$, while for
    \[\J' = \begin{matrix}
        &\{2,4\}& \\
        \{3,4\}&&\{2,3\} \\
        &\{3,4\}&
    \end{matrix}\]
    we compute
    \[R(\J') = \begin{matrix}
        &\{2,4\}& \\
        \{1,4\}&&\{3,4\} \\
        &\{3,4\}&
    \end{matrix} \, .\]
    Here we have written the juggling patterns with vertex 0 at the top and the other ones clockwise in order.
\end{example}

Observe that the ordering from \eqref{eq:orderforsets} is also well defined for subsets of a finite totally ordered set, which must be isomorphic to $[N]$ for some $N$.

\begin{lemma}\label{lem:stupid}
Let $A$ and $B$ be subsets of a totally ordered finite set $T$ and suppose they have the same cardinality $d$. If $A \le B$, then $B^c \le A^c$.
\end{lemma}
\begin{proof}
	We prove this result by induction on $d$. The base case $d = 0$ is clearly trivial. When $d \ge 1$, let $A = \{a_1 < a_2 < \dots < a_d\}$ and $B = \{b_1 < b_2 < \dots < b_d\}$. Since $A \le B$, we also have $A \mysetminus \{a_d\} \le B \mysetminus \{b_d\}$. These sets have cardinality $d-1$ so by inductive hypothesis we get $A^c \cup \{a_d\} \ge B^c \cup \{b_d\}$. Let us call them $X$ and $Y$ respectively and let $x_i$ and $y_i$ be their (ordered) elements, where $i \in [m+1]$ and $m = \vert T \vert - d$. The elements $a_d \in X$ and $b_d \in Y$ are of the form $x_t$ and $y_s$ for some integers $t \le s$, since $a_d \le b_d$.
    Let $x'_i$ and $y'_i$ denote the (ordered) elements of $A^c = X \mysetminus \{a_d\}$ and $B^c = Y \mysetminus \{b_d\}$. By definition they are
    \[\begin{cases}
        x'_i = x_i & 1 \le i < t \\ x'_i = x_{i+1} & t \le i \le m
    \end{cases} \qquad \qquad \begin{cases}
        y'_i = y_i & 1 \le i < s \\ y'_i = y_{i+1} & s \le i \le m
    \end{cases} \, .\]
    Comparing them, if $t=s$ we can simply remove them from $X$ and $Y$, otherwise we have
    \[\begin{cases}
        y'_i = y_i \le x_i = x'_i & 1 \le i < t \\
        y'_i = y_i \le x_i < x_{i+1} = x'_i  & t \le i < s \\
        y'_i = y_{i+1} \le x_{i+1} = x'_i & s \le i \le m \, .
    \end{cases}\]
    In either case the inductive step is proved.
\end{proof}

\begin{corollary}\label{cor:GNorder}
    The map $R \colon JP(n,2n) \longrightarrow JP(n,2n)$ is order-preserving.
\end{corollary}

\begin{proof}
    Observe that $x \ge y$ implies $\tilde{y} \ge \tilde{x}$. Let $\J \le \J'$ be $(n, 2n)$-juggling patterns; since we have $J_i \ge J'_i$ for all $i$, we obtain $(RJ)_{-i} \ge (RJ')_{-i}$ from Lemma \ref{lem:stupid}, that is, $R\J \le R\J'$.
\end{proof}

\begin{definition}
    A $(n,2n)$-juggling pattern $\J$ is called \emph{symplectic} if $\J = R\J$, or equivalently if $p_\J \in X(n,2n)^{sp}$. Their set will be denoted by $JP(n,2n)^{sp}$.
\end{definition}

\begin{lemma}\label{lem:goodorbits2}
    A cell $C_\J$ in $X(n,2n)$ intesects $X(n,2n)^{sp}$ if and only if its juggling pattern point $p_\J$ is in $X(n,2n)^{sp}$.
\end{lemma}
\begin{proof}
    One implication is obvious, since $C_\J = G \cdot p_\J$. For the other one, let $g \in G$ be such that $g \cdot p_\J \in X(n,2n)^{sp}$. Then we have
    \[g \cdot p_\J = \tau (g \cdot p_\J ) = \sigma(g) \cdot ( \tau p_\J ) = \sigma(g) \cdot (p_{R\J}) \, .\]
    The same point is in the $G$-orbit of two juggling pattern points, so $p_{R\J}$ must coincide with $p_\J$, that is, $\J$ must be symplectic.
\end{proof}

Lemma \ref{lem:goodorbits2} and Proposition \ref{prop:orbitintersection} imply that the $G^{sp}$-orbits in $X(n,2n)^{sp}$ are parametrized by the symplectic $(n,2n)$-juggling patterns. This, together with \eqref{eq:goodorbits1} allows us to define \emph{symplectic cells}:

\begin{definition}
    If $\J$ is a symplectic $(n,2n)$-juggling pattern, its $G^{sp}$-orbit will be denoted by $C_\J^{sp}$.
\end{definition}

These orbits are again affine cells, by \cite[Theorem A]{ours}. More will be said about them in Sections \ref{sec:torusaction2} and \ref{sec:mainresults}.

\subsection{Bounded affine permutations}
We will now introduce a set of permutations of the integers which is isomorphic to the set of juggling patterns, and which will later allow us to make key computations needed for the main proofs of this paper. See also \cite[2.7]{karpman}.

\begin{definition}
    An $(k,N)$\emph{-affine permutation} is a bijection of the integers $f \colon \Z \longrightarrow \Z$ such that $f(i+N) = f(i)+N$ for all $i \in \Z$ and that
    \[\sum_{i=1}^N (f(i)-i) = kN \, .\]
\end{definition}

Observe that the first condition implies that any such sum over $N$ consecutive numbers results in $kN$. We denote the set of $(k,N)$-affine permutations with $A^k_N$. The reader will notice that the composition of two $(k,N)$-affine permutations is not another $(k,N)$-affine permutation, unless $k=0$. Moreover, $A^0_N$ with such a product forms a group, which is a Coxeter group of affine type $A_{N-1}^{(1)}$. The set $S$ of its simple reflections consists of the affine permutations $s_0, s_1, \dots, s_{N-1}$, where $s_i$ switches $i+tN$ and $i+1 +tN$ for all $t \in \Z$. If $j$ is any integer, we write $s_j$ for the simple reflection $s_i$ with $j \equiv_N i$. Conjugates of simple reflections are called reflections and are denoted by $(i,j)$, where it is understood that such a permutations switches $i+tN$ and $j+tN$ for all $t \in \Z$.

\smallskip

While $A^k_N$ is not a group, there is a bijection $A^0_N \longrightarrow A^k_N$ given by
\[g \mapsto g \circ \text{id}_k\]
where $\text{id}_k(i) = i+k$. Since $(A^0_N,S)$ is a Coxeter group, this allows us to transfer its Bruhat order to $A^k_N$ in the expected way: we have $g \circ \text{id}_k \le g' \circ \text{id}_k$ whenever $g \le g'$. For $f \in A^k_N$, its length $\ell(f)$ is defined as the cardinality of the set:
\begin{equation}\label{eq:lengthset}
L(f) \coloneqq \{(x,y) \in [0,N-1] \times \Z \, \, \vert \, \, x < y \, \land \, f(x) > f(y)\} \, \, .\end{equation}
If $f = g \circ \text{id}_k$, its length as a $(k,N)$-affine permutation coincides with the length of $g$ as an element of the Coxeter group $(A^0_N,S)$ \cite{BjornerBrenti}.
\begin{definition}
    A $(k,N)$-affine permutation $f$ is \emph{bounded} if it satisfies
    \[i \le f(i) \le i+N\]
    for all $i \in \Z$.
\end{definition}
These permutations form a lower order ideal in $A^k_N$, which we will call $\mathcal{B}_{(k,N)}$, and it is in order-preserving bijection with the poset of $(k,N)$-juggling patterns. In \cite{ML1} the correspondence with Grassmann necklaces is shown. In our setting, precomposing with the isomorphism between the postes of juggling patterns and Grassmann necklaces, we get the following: given a $(k,N)$-juggling pattern $\J$, the corresponding $(k,N)$-bounded affine permutation is defined by:
\begin{equation}\label{eq:boundedfromjuggling}
f_\J(a) \coloneqq \begin{cases}
    a & \qquad N \notin J_{\overline{a}} \, , \\
     a+N+1-b & \qquad J_{\overline{a+1}}=s(J_{\overline{a}} \mysetminus \{N\}) \cup \{b\}
\end{cases}\end{equation}
where $\overline{a}$ is the residue class of $a \in \Z$ modulo $N$. Conversely, to get a juggling pattern $\J_f$ from a bounded affine permutation $f$, we start with an integer $a$ and consider the set
\begin{equation}\label{eq:jugglingfrombounded}
    \left\{a-f(b) \, \vert \, b<a \land f(b) \ge a \right\} \subset \Z \, .
\end{equation}
Then $\left(\J_f \right)_{\overline{a}}$ is given by the representatives in $[N]$ of the cosets of the elements in \eqref{eq:jugglingfrombounded}. Observe that it really only depends on $\overline{a}$ since $f$ is bounded and affine. Furthermore, any affine cell $C_\J$ has dimension equal to $\ell(f_\J)$ \cite[Theorem 7.5]{ML1}.

\begin{remark}
    Given an affine permutation $f$, we have $f \circ (i,j) = (f(i),f(j)) \circ f$. Moreover, if both $f$ and $f'=f \circ (i,j)$ are bounded, then $\vert \, j-i \, \vert < N$. Otherwise, assuming $f$ is bounded, the following chain of inequalities
    \[i \le f(i) = f'(j) \le i + N < j\]
    proves that $f'$ cannot be bounded as well.
\end{remark}

\begin{lemma}\label{lem:mutatedpermutations}
    Let $f \le f'$ be two $(k,N)$-bounded affine permutations such that $f = g \circ \text{id}_k$ and $f' = g \circ t \circ \text{id}_k$, for some $g \in A^0_N$ and $t = (i,j)$ with $i < j$. Then, the juggling patterns $\J_f$ and $\J_{f'}$ differ only on the vertices from $i-k+1$ to $j-k$, and the difference is that $(\J_f)_{i-k+s}$ and $(\J_{f'})_{i-k+s}$ contain, up to equivalence modulo $N$, $i-k+s-f(i-k)$ and $i-k+s-f'(i-k)$ respectively.
\end{lemma}

\begin{proof}
    For $b \in \Z$, we have that $f(b) = f'(b)$ if and only if $b \not \equiv_{N} i-k, \, j-k$, and that $f(i-k)=f'(j-k)$, $f(j-k)=f'(i-k)$.
    In addition, since $f$ and $f'$ are bounded and affine we know that $f(i-k), f'(i-k) \ge j-k$, thus for $i-k+1 \le a \le j-k$, the sets in \eqref{eq:jugglingfrombounded} for $f$ and $f'$ contain $a-f(i-k)$ and $a-f'(i-k)$ respectively. We conclude by observing that they have different residue class modulo $N$, once again because $f$ and $f'$ are affine.
\end{proof}

\begin{example}\label{ex:running2}
    The bounded affine permutation corresponding to the juggling patterns from Example \ref{ex:running1} are
    \begin{align*}
        f_\J: 0& \mapsto 1 \quad 1 \mapsto 3 \quad 2 \mapsto 4 \quad 3 \mapsto 6 \\
        f_{\J'}: 0& \mapsto 3 \quad 1 \mapsto 1 \quad 2 \mapsto 4 \quad 3 \mapsto 6 \\
        f_{R(\J')}: 0& \mapsto 1 \quad 1 \mapsto 3 \quad 2 \mapsto 6 \quad 3 \mapsto 4
    \end{align*}
    and the first two satisfy $f_{\J'} = f_\J \circ (0,1)$, therefore we have $f_{\J'} = g_\J \circ (2,3) \circ \text{id}_2$ where $f_\J = g_\J \circ \text{id}_2$. The two juggling patterns differ only on vertex 1, where the difference is $4 \in J_1$ and $2 \in J'_1$.
\end{example}

\subsection{Symplectic bounded affine permutations} \label{subsec:sympl-permutations}
Since there is a bijection between the poset of $(n, 2n)$-juggling patterns and the poset of $(n, 2n)$-bounded affine permutations, we can draw the following diagram.
\[\begin{tikzcd}
    JP(n, 2n) \arrow[r, leftrightarrow] \arrow[d, leftrightarrow, "R"] & \mathcal{B}_{(n, 2n)} \\ JP(n, 2n) \arrow[r, leftrightarrow] & \mathcal{B}_{(n, 2n)}
\end{tikzcd}\]
We complete the square by composing the maps, and obtain a symmetry which is once again order-preserving. In order to explicitly describe it, we define a function $\mathcal{B}_{(n, 2n)} \longrightarrow \mathcal{B}_{(n, 2n)}$, and prove that it is the one we are looking for. Let us call it $R$ as well. This map is given by
\begin{equation}\label{eq:Rforpermutations}
R(f)(i) \coloneqq 2n-f(-i-1)-1 = 4n-f(2n-i-1)-1
\end{equation}
as explained in \cite[Section 3]{karpman}.

\begin{prop}
The image of any $(n,2n)$-bounded affine permutation under $R$ is an $(n,2n)$-bounded affine permutation.
\end{prop}

\begin{proof}
    Fix $f \in \mathcal{B}_{(n,2n)}$. First we show that $Rf$ is an affine permutation:
    \begin{itemize}[\phantom{a}]
        \item \[Rf(i+2n) = 2n - f(-i-2n-1)-1 = 4n -f(-i-1)-1 = Rf(i)+2n;\]
        \item \begin{gather*}\sum_{i=1}^{2n} (Rf(i)-i) = \sum_{i=1}^{2n} (2n-f(-i-1)-i-1) =  4n^2 -\sum_{i=1}^{2n} (f(-i-1)+i+1) = \\
        = 4n^2 - \sum_{j=-2n-1}^{-2} (f(j)-j) = 4n^2 - 2n^2 = 2n^2 \, \, .\end{gather*}
    \end{itemize}
    Now we prove that $Rf$ is bounded: for all $i \in \Z$, we have
    \[-i-1 \le f(-i-1) \le 2n-i-1\]
    since $f$ is bounded. We change signs and add $2n-1$ to obtain
    \[i \le Rf(i) \le i+2n \quad \forall \, i \in \Z\]
    which concludes the proof.
\end{proof}

\begin{lemma}\label{lem:correctsymmetry}
For any $(n,2n)$-juggling pattern $\J$, the equality $R(f_\J) = f_{R(\J)}$ holds.
\end{lemma}

\begin{proof}
    By definition, if $\J$ is such that $2n \in J_a$ for some vertex $a$ and that $J_{a+1}=s(J_a \mysetminus \{2n\}) \cup \{b\}$ with $b \ne 1$, then $2n \in (R\J)_{-a-1}$ and $(R\J)_{-a}=s((R\J)_{-a-1} \mysetminus \{2n\}) \cup \{\tilde{b}+1\}$ with $\tilde{b}+1 \ne 1$. This is true because no element of this set is of the form $\tilde{x}$ with $x \in J_{a}$, and because it has the right cardinality. The opposite also holds, by switching $\J$ and $R\J$. Remember that the bounded affine permutation $f_\J$ corresponding to $\J$ is defined in \eqref{eq:boundedfromjuggling}. We are going to look at the images of the elements in $[2n]$ under $f_{R_\J}$ and compare them to their images under $Rf_\J$. In the subsequent chains of equivalent conditions we use the fact that $1 \in J_{a+1}$ if and only if $2n \in J_a$ and $b=1$. This means that, if $1$ is not in $J_{a+1}$ then either $2n \notin J_a$, or $2n \in J_a$ and $b \in [2,2n]$. First,
\begin{gather*}
    f_{R\J}(a) = a \iff 2n \notin (R\J)_a \iff 1 \in J_{-a} \iff \\ \iff 2n \in J_{-a-1} \, \land \, J_{-a}=s(J_{-a-1} \mysetminus \{2n\}) \cup \{1\} \iff \\
    \iff f_\J(-a-1) = -a-1+2n \iff Rf_\J(a)=a \, .
    \end{gather*}

\noindent Moreover,

\begin{gather*}
    f_{R\J}(a) = a+2n \iff 2n \in (R\J)_a \, \land \, (RJ)_{a+1}=s \bigl( (R\J)_a \mysetminus \{2n\}\bigr) \cup \{1\} \iff \\ \iff 1 \in (R\J)_{a+1} \iff 2n \notin J_{-a-1} \iff \\ \iff f_\J(-a-1) = -a-1 \iff Rf_\J(a) = a+2n \, .
\end{gather*}

\noindent Finally, considering $b \in [2,2n]$ and therefore $\tilde{b}+1 \in [2,2n]$,

\begin{gather*}
    f_{R\J}(a) = 2n+a+1-b \iff \\ \iff 2n \in (R\J)_a \, \land \, (R\J)_{a+1} =s \bigl( (R\J)_a \mysetminus\{2n\} \bigr) \cup \{b\} \iff \\ 
    \iff 2n \in J_{-a-1} \, \land \, J_{-a} = s(J_{-a-1} \mysetminus\{2n\}) \cup \{\tilde{b}+1\} \iff \\ \iff f_\J(-a-1) = 2n+(-a-1)+1 -(\tilde{b}+1) \iff \\ \iff Rf_\J(a) = 2n+a+1-b \, .
\end{gather*}
\end{proof}

\begin{example}\label{ex:running3}
    Let us apply \eqref{eq:Rforpermutations} to the permutations $f_\J$ and $f_{\J'}$ from Example \ref{ex:running2}. We obtain
    \begin{align*}
        R(f_\J) = f_\J: 0& \mapsto 1 \quad 1 \mapsto 3 \quad 2 \mapsto 4 \quad 3 \mapsto 6 \\
        R(f_{\J'}): 0& \mapsto 1 \quad 1 \mapsto 3 \quad 2 \mapsto 6 \quad 3 \mapsto 4
    \end{align*}
    which agrees with Example \ref{ex:running1}, where we noticed that $\J$ is symplectic while $\J'$ is not. In Example \ref{ex:running2} we computed $f_{R(\J')}$, which indeed coincides with $R(f_{\J'})$. 
\end{example}

From Corollary \ref{cor:GNorder} we obtain:
\begin{corollary}
        The map $R \colon \mathcal{B}_{(n,2n)} \longrightarrow \mathcal{B}_{(n,2n)}$ is order-preserving.
\end{corollary}

\begin{lemma}\label{lem:length}
    For any $(n, 2n)$-bounded affine permutation $f$, we have $\ell(f) = \ell(Rf)$.
\end{lemma}

\begin{proof}
First, we know that $\ell(f)$ is the cardinality of the set $L(f)$, as defined in \eqref{eq:lengthset}. We claim that the following function is a bijection between $L(f)$ and $L(Rf)$:
\[\Phi \colon (x,y) \longmapsto \begin{cases}(2n-y-1, 2n-x-1) & \text{if} \, \, y \in [0,2n-1] \, ; \\ (4n-y-1, 4n-x-1) & \text{if} \, \, y \in [2n,4n-1] \, . \end{cases}\]
Since $f$ is bounded, i.e. satisfies $x \le f(x) \le x + 2n$ for all $x$, a pair $(x,y) \in L(f)$ must satisfy
\begin{equation}\label{eq:ij-inequality}
0 \le x < y \le f(y) < f(x) \le x+2n \le 4n-1 \, .
\end{equation}
We change sign and add $2n-1$ to all the terms of the chain, to obtain
\[
	2n-1 \ge 2n-x-1 > 2n-y-1 \ge 2n-f(y)-1 > 2n-f(x)-1 \ge -x-1 \ge -2n \, .
\]
Notice that the first entry of the pair $\Phi(x,y)$ is always in $[0, 2n-1]$ and is strictly lower than the second entry. Next we compare their images under $Rf$:
\begin{align*}
    Rf(2n-y-1) = 4n-f(y)-1 > 4n-f(x)-1 = Rf(2n-x-1) \qquad & \, \, \text{if} \, \, y \in [0,2n-1] \, ; \\
    Rf(4n-y-1) = 2n-f(y)-1 > 2n-f(x)-1 = Rf(4n-x-1) \qquad & \, \, \text{if} \, \, y \in [2n,4n-1] \, .
\end{align*}
Therefore $\Phi(x,y) \in L(Rf)$ whenever $(x,y) \in L(f)$. Lastly we observe that $\Phi \circ \Phi (x,y) = (x,y)$, so $\Phi$ is bijective.
\end{proof}

\begin{remark}
    By \cite[Lemma 3.2]{ours}, the image under $\tau$ of an affine cell $C_\J$ in $X(n,2n)$ is the affine cell $C_{R\J}$.
\end{remark}

\begin{lemma}\label{lem:tauforcells}
Let $\J$ be a symplectic $(n,2n)$-juggling pattern. Then the coordinates in $C_\J$ of a point in $C_\J^{sp}$ satisfy
\[u^{(a)}_{i,j} = (-1)^{i+j+1} \cdot u^{(-a)}_{\tilde{j}, \tilde{i}}\, .\]
\end{lemma}

\begin{proof}
    Suppose first that $\J$ is any $(n,2n)$-juggling pattern. Let $V=(V_a)_a \in C_\J$ with affine coordinates $\bigl( u^{(a)}_{i,j} \, \vert \, a \in \Z_{2n}$, $j \in J_a , j<i \notin J_a \bigr)$. We claim that the affine coordinates of $\tau(V) \in C_{R\J}$ are
    \[\Bigl( (-1)^{x+y+1} \cdot u^{(-b)}_{\tilde{y}, \tilde{x}} \, \vert \, b \in \Z_{2n}, y \in (R\J)_b, y<x \notin (R\J)_b \Bigr).\]
    Such coordinates define a point in $X(n,2n)$ which sits inside $C_{R\J}$; let us call it $W$. We want to prove $W=\tau(V)$ by showing that the generators for $V$ and for $W$ defined in \eqref{eq:coordinatebasis} pair trivially via $\Omega$. Fix a vertex $a$ and two elements $j \in J_a, i \in (R\J)_{-a}$. The symplectic form evaluated in $v_j^{(a)}$, the $j$-th generator of $V_a$, and in $w_i^{(-a)}$, the $i$-th generator of $W_{-a}$, is
    \begin{gather*}
        (v_j^{(a)}, w_i^{(-a)})_\Omega = \bigl( \, e^{(a)}_j + \sum_{\substack{x>j \\ x \notin J_a}} u^{(a)}_{x,j} e^{(a)}_{x} \, , \,  e^{(-a)}_i + \sum_{\substack{y>i \\ y \notin (R\J)_{-a}}} (-1)^{i+y+1} \cdot u^{(a)}_{\tilde{i},\tilde{y}} e^{(-a)}_{y} \, \bigr)_\Omega = \\ = (e_j^{(a)} + r_j, e_i^{(-a)} + r_i)_\Omega
    \end{gather*}
    where we denoted $v_j^{(a)}-e_j^{(a)}$ by $r_j$ and $w_i^{(-a)}-e_i^{(-a)}$ by $r_i$. By bilinearity of $(\cdot,\cdot)_\Omega$, it equals
    \[(e_j^{(a)},e_i^{(-a)})_\Omega +(e_j^{(a)},r_i)_\Omega +(r_j,e_i^{(-a)})_\Omega +(r_j,r_i)_\Omega \, .\]
    The first summand is zero since $j \in J_a$ and $i \in R(J_a)$. The second term can be nonzero only if $\tilde{j}$ appears as one of the indices $y$, i.e. if and only if $\tilde{j} >i$, in which case it is equal to $(-1)^{j+1}(-1)^{i+j} \cdot u^{(a)}_{\tilde{i}, j} = (-1)^{i+1} \cdot u^{(a)}_{\tilde{i}, j}$. The same happens for the third term: it is again nonzero only if $\tilde{j} >i$, when it is equal to $(-1)^i \cdot u^{(a)}_{\tilde{i},j}$. Therefore they are opposites and cancel out. Lastly, the fourth term is zero because the basis vectors $e^{(a)}_x$ with $x \notin J_a$ pair trivially with the vectors $e^{(-a)}_y$ for $y \notin R(J_a)$. We conclude that $W = \tau(V)$. If, in addition, $\J$ is symplectic and $V = \tau(V)$, we obtain the desired equations.
\end{proof}

\begin{corollary}\label{cor:symplcoordinates}
    If $\J$ is symplectic, then the affine coordinates for $C^{sp}_\J$ are parametrized by equivalence classes of the same triples $(a,i,j)$ as before, but under the following equivalence relation:
    \[
    (a,i,j) \overset{sp}{\sim} (a+1,i+1,j+1) \quad \land \quad (a,i,j) \overset{sp}{\sim} (-a, \tilde{j}, \tilde{i}) \, .\]
\end{corollary}

The next definition coincides with the one given by Karpman \cite[Section 3]{karpman}.

\begin{definition}
    A bounded affine permutation $f \in \mathcal{B}_{(n,2n)}$ is called \emph{symplectic} if $Rf=f$, or equivalently if
        \begin{equation}\label{eq:symplecticpermutation}
        f(2n-i-1) - (2n-i-1) + f(i)-i = 2n \quad \forall \, i \in [0,2n-1] \, .
    \end{equation}
\end{definition}

\begin{remark}
    By Lemma \ref{lem:correctsymmetry}, an $(n,2n)$-juggling pattern is symplectic if and only if its corresponding bounded affine permutation is.
\end{remark}

Since $R$ extends without issue to $A^n_{2n}$ and its image is again $A^n_{2n}$, if we precompose it with id$_{n}$ and postcompose with id$_{-n}$ we find the group automorphism $R^0 \colon A^0_{2n} \longrightarrow A^0_{2n}$ given by
\[(R^0g)(a) = -g(-a-1)-1 \, ,\]
which has order 2. The image of the affine reflection $(i,j)$ under $R^0$ is $(-j-1, -i-1)$, so in particular
\[R^0 (s_i) = s_{-i-2} \, \, .\]
The fixed simple reflections are $s_{-1}$ and $s_{n-1}$. A permutation and its image commute if and only if their product is $R^0$-fixed, so for any $i \in \Z_{2n} \mysetminus \{\overline{-1},\overline{n-1}\}$, $s_i$ and $s_{-i-2}$ are not fixed but their product is.

\begin{definition}
    We define the subgroup
    \[C^0_{n+1} \coloneqq \{g \in A^0_{2n} \, \vert \, R^0(g)=g\} \, .\]
\end{definition}

\begin{prop}\label{prop:affinesymmetry}
$C^0_{n+1}$ is a Coxeter group of affine type $C^{(1)}_n$, with simple reflections $r_{-1} \coloneqq s_{-1}$, $r_{n-1} \coloneqq s_{n-1}$ and $r_i = s_i \circ s_{-i-2}$ for $i \in [0,n-2]$. We denote by $S_C$ this set of generators.
\end{prop}
\begin{proof}
    First we prove that $S_C$ generates the fixed-point subgroup: we do so by showing that any $g \in C^0_{n+1}$ has an expression in terms of the $r_i$, by induction on the length of $g$. For $\ell(g)=0$ or 1, $g$ is either the identity or a simple reflection fixed by $R^0$, i.e. $g = s_{-1} = r_{-1}$ or $g=s_{n-1}=r_{n-1}$. If $\ell(g)=2$ then $g=s_i \circ s_j$ for some $i,j \in [0,2n-1]$. Since $g = R^0g = s_{-i-2} \circ s_{-j-2}$, then either $s_i=s_{-i-2}$ and $s_j=s_{-j-2}$, which implies $g=s_{-1} \circ s_{n-1}=r_{-1}\circ r_{n-1}$, or $s_j=s_{-i-2}$ so $g=r_i$. Now for the inductive step, let $g=u \circ s_i$ be a minimal length expression for $g$, that is, $\ell(u)=\ell(g)-1$. Then $g=R^0g=R^0u \circ s_{-i-2}$ is another minimal length expression for $g$. If $i \in \{-1, n-1\}$, then $s_i = r_i$ and $u = R^0u$, then by induction $u$ is the product of elements in $S_C$, therefore $g$ is as well. Otherwise, $s_{-i-2} \ne s_i$ therefore, by the exchange condition, there exists a word $v$ of length $\ell(g)-2$ such that $g = v \circ s_i \circ s_{-i-2}$, so we conclude by applying the inductive hypothesis to $v$ since it is also $R^0$-fixed.

    \smallskip
    
    The order of $r_{-1} \circ r_0=s_{-1} \circ s_0s_{-2}$ is 4, since $-2 \mapsto 0$, $0 \mapsto 1$, $1 \mapsto -1$ and $-1 \mapsto -2$. For a similar reason $r_{n-2} \circ r_{n-1}$ has order 4 as well. For $i \in [0, n-3]$, the order of $r_i \circ r_{i+1} =s_i \circ s_{-i-2} \circ s_{i+1} \circ s_{-i-3} = (s_i \circ s_{i+1} ) \circ (s_{-i-3} \circ s_{-i-2})$ is the least common multiple of the orders of the two commuting factors, which is 3. Lastly, for $\vert i - j \vert \ge 2$, we have $r_i \circ r_j = r_j \circ r_i$, since $s_i$, $s_{-i-2}$, $s_j$ and $s_{-j-2}$ all commute with one another by the type $A$ Coxeter relations. There is no other relation in $C^0_{n+1}$ that does not stem from these, because simple reflections in type $C$ are either simple reflections in type $A$ or product of two commuting simple reflections, therefore any non-Coxeter relation in type $C$ would imply one in type $A$.
\end{proof}

\begin{corollary}\label{cor:transpositionsintypeC}
    Transpositions in this Coxeter system are either reflections $(i,-i-1)$ or pairs of reflections of the form $(i,j)(-j-1,-i-1)$.
\end{corollary}

\section{Torus action} \label{sec:torusaction1}

Quiver Grassmannians for nilpotent representations of the equioriented cycle, such as the varieties $X(k,N)$, are equipped with the action of certain algebraic tori \cite{LaniniPuetz}. In this section we will recall such an action and see that it provides information about the $G$-orbit closure inclusion order.

\begin{definition}
    Given a quiver $Q$ and a $Q$-representation $M$, a \emph{basis} $B$ for it is a collection of bases $B_i$ for each of its vector spaces $M_i$, $i \in Q_0$. The \emph{coefficient quiver} $Q(M,B)$ is the quiver whose vertices are the elements of $B$, with an arrow between $b \in B_i$ and $b' \in B_j$ if there is an arrow $\alpha \colon i \longrightarrow j \in Q_1$ such that the coefficient of $b'$ for the expression of $M_\alpha (b)$ in the basis $B_j$ is nonzero.
\end{definition}

Let us consider the coefficient quiver $Q(U_{[N]},SB)$ for the standard basis $SB$ of $\C^N$ on every vertex, i.e. $SB_i= \{e^{(i)}_j\}$. For $N=4$, the coefficient quiver is as follows.

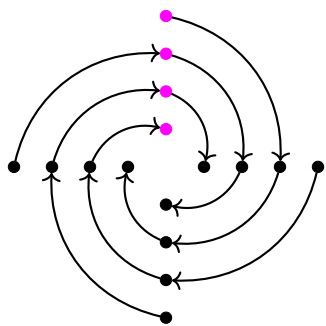
\begin{center}
    \begin{tikzpicture}
     \begin{scope}[every node/.style={circle, draw=black!100, fill=black!100, 
    very thin, inner sep = 0pt, minimum size = 1.5mm}]
    \node[fill=pink!100, draw=pink!100] (01) at (0,2) {};
    \node[fill=pink!100, draw=pink!100] (02) at (0,1.5) {};
    \node[fill=pink!100, draw=pink!100] (03) at (0,1) {};
    \node[fill=pink!100, draw=pink!100] (04) at (0,0.5) {};
    \node (11) at (2,0) {};
    \node (12) at (1.5,0) {};
    \node (13) at (1,0) {};
    \node (14) at (0.5,0) {};
    \node (21) at (0,-2) {};
    \node (22) at (0,-1.5) {};
    \node (23) at (0,-1) {};
    \node (24) at (0,-0.5) {};
    \node (31) at (-2,0) {};
    \node (32) at (-1.5,0) {};
    \node (33) at (-1,0) {};
    \node (34) at (-0.5,0) {};
\end{scope}

\begin{scope}[every edge/.style= 
              {draw=black,thick}]
\path [->] (01) edge[bend left=40] (12);
\path [->] (02) edge[bend left=40] (13);
\path [->] (03) edge[bend left=40] (14);
\path [->] (11) edge[bend left=40] (22);
\path [->] (12) edge[bend left=40] (23);
\path [->] (13) edge[bend left=40] (24);
\path [->] (21) edge[bend left=40] (32);
\path [->] (22) edge[bend left=40] (33);
\path [->] (23) edge[bend left=40] (34);
\path [->] (31) edge[bend left=40] (02);
\path [->] (32) edge[bend left=40] (03);
\path [->] (33) edge[bend left=40] (04);

\end{scope}

    \end{tikzpicture}
    \end{center}
In such a diagram, the $N$ vertices in the topmost line (in pink) are the basis vectors for $U_{[N]}^{(0)}$, then the next line clockwise is the basis for $U_{[N]}^{(1)}$, and so on. Moreover, the vectors of each basis are ordered from the outside in. We denote by $b_{j,p}$ the $p$-th element of the segment ending with $e^{(j)}_N$, that is, $b_{j,p} = e^{(j+p)}_p$. In this case we say that the segment ends at vertex $j$ in $\Delta_N$. Here $p \in [N]$ and we consider the representative of the class $j$ that lies in $[0, N-1]$.

\begin{center}
    \begin{tikzpicture}
     \begin{scope}[every node/.style={circle, draw=black!100, fill=black!100, 
    very thin, inner sep = 0pt, minimum size = 1.5mm}]
    \node[fill=pink!100, draw=pink!100] (01) at (0,2) {};
    \node (02) at (0,1.5) {};
    \node (03) at (0,1) {};
    \node (04) at (0,0.5) {};
    \node (11) at (2,0) {};
    \node[fill=pink!100, draw=pink!100] (12) at (1.5,0) {};
    \node (13) at (1,0) {};
    \node (14) at (0.5,0) {};
    \node (21) at (0,-2) {};
    \node (22) at (0,-1.5) {};
    \node[fill=pink!100, draw=pink!100] (23) at (0,-1) {};
    \node (24) at (0,-0.5) {};
    \node (31) at (-2,0) {};
    \node (32) at (-1.5,0) {};
    \node (33) at (-1,0) {};
    \node[fill=pink!100, draw=pink!100] (34) at (-0.5,0) {};
\end{scope}

\begin{scope}[every edge/.style= 
              {draw=black,thick}]
\path [->] (01) edge[bend left=40,draw=pink] (12);
\path [->] (02) edge[bend left=40] (13);
\path [->] (03) edge[bend left=40] (14);
\path [->] (11) edge[bend left=40] (22);
\path [->] (12) edge[bend left=40,draw=pink] (23);
\path [->] (13) edge[bend left=40] (24);
\path [->] (21) edge[bend left=40] (32);
\path [->] (22) edge[bend left=40] (33);
\path [->] (23) edge[bend left=40,draw=pink] (34);
\path [->] (31) edge[bend left=40] (02);
\path [->] (32) edge[bend left=40] (03);
\path [->] (33) edge[bend left=40] (04);

\end{scope}
    \end{tikzpicture}
    
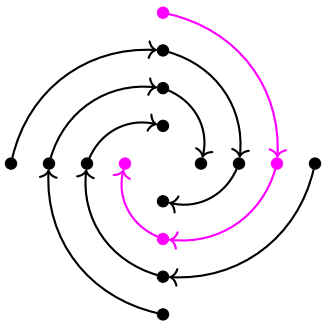
    The segment ending at vertex 3.
    \end{center}
Given an integer-valued function $\mathbf{wt}$ on a basis $B$ of a $Q$-representation $M$, called a \emph{grading} or \emph{weight function}, there is a corresponding $\C^*$-action on the total space
\[\bigoplus_{i \in Q_0} M_i\]
of $M$ given on the basis by
\[z \cdot b \coloneqq z^{\mathbf{wt}(b)} b \, .\]
This action, under appropriate assumptions, extends to quiver Grassmannians for $M$ \cite{Pue25}. Our grading of choice is
\[\mathbf{wt} \colon b_{j,p} \longmapsto p \, ,\]
which meets the requirements to be \emph{attractive} \cite[Definition 1.17]{Jaco}, i.e. it satisfies
\[\mathbf{wt}(e^{(i)}_m) < \mathbf{wt}(e^{(i)}_{m'})\]
whenever $m < m'$, and the weight increases by the same amount from the source to the target of any arrow of the coefficient quiver. In this case we have $\mathbf{wt}(b_{j,p+1}) = \mathbf{wt}(b_{j,p})+1$ so the increase is equal to 1. The action on $M$ corresponding to an attractive grading extends to $\Gr_{\ubar{d}}(M)$ for any dimension vector $\ubar{d}$ \cite{LaniniPuetz, giovanni, Pue20}. Thus we have the following $\C^*$-action on $X(k,N)$:
\begin{equation}\label{eq:oldC*action}
	z \cdot b_{j,p} = z^p b_{j,p} \, \, .
\end{equation}
Its fixed points are the coordinate points $p_\J$, and its attracting set
\[\{ V \in X(k,N) \, \vert \, \lim_{z \rightarrow 0} z \cdot V = p_\J \}\]
to a fixed point $p_\J$ is the $G$-orbit $C_\J$ \cite{ML1}.
Let us also recall from \cite{LaniniPuetz} another action: let $T$ be an $(N+1)$-dimensional algebraic torus whose elements we denote by $t=(z, \gamma_0, \gamma_1, \dots, \gamma_{N-1})$. The $T$-action on the basis is given by
\[t \cdot b_{j,p} \coloneqq z^p \gamma_j b_{j,p}\]
and, once again, it extends to the variety. Its fixed points are also precisely the points $p_\J$ \cite[Theorem 5.14]{LaniniPuetz}. We refer to \cite{Pue20} for further details. This action makes $X(k,N)$ into a GKM variety: we give a definition equivalent, for our case, to the one given in the seminal paper \cite{GKM}. See also \cite{LaniniPuetz}.

\begin{definition}
Let $X$ be a projective variety and $T$ an algebraic torus acting on $X$. Then $X$ is a \emph{GKM variety} if
\begin{itemize}
    \item the $T$-action on $X$ is skeletal, i.e. it has finitely many fixed points and 1-dimensional orbits;
    \item the rational cohomology of $X$ vanishes in odd degree.
\end{itemize}
\end{definition}

Since $X(k,N)$ is GKM, then the $T$-action is \emph{locally linearizable}, i.e. for any 1-dimensional orbit $O$ there exists a $T$-action on $\C\P^1$ and a $T$-equivariant isomorphism $\overline{O} \cong \C\P^1$. Thus, the closure of each 1-dimensional orbit is isomorphic to $\C\P^1$ and the boundary consists of two 0-dimensional orbits, i.e. fixed points. Notice that $t = (z, \gamma_0, \gamma_1, \dots, \gamma_{N-1}) \in T$ decomposes as
\[(z, 1, 1, \dots, 1) \cdot (1, \gamma_0, \gamma_1, \dots, \gamma_{N-1}) \, .\]
The first factor acts as $z \in \C^*$ on $X(k,N)$ while the latter, when written in the standard basis, becomes a tuple of diagonal matrices that form an element of $G$. So the affine cells are $T$-stable. In particular, each 1-dimensional orbit is contained in a cell $C_\J$, so the two fixed points $p_\J$ and $p_{\J'}$ it connects must be comparable in the $G$-orbit closure inclusion order.
\begin{definition}
On the level of juggling patterns, we say that $\J$ and $\J'$, are connected by a \emph{mutation} $\mu \colon \J' \longrightarrow \J$ if $\J' \ge \J$ and $p_{\J'}$ and $p_\J$ are the closure points of a 1-dimensional $T$-orbit.
\end{definition}

In a Coxeter system $(W,S)$, two group elements $w$ and $w'$ are in Bruhat order relation if there exists a sequence of reflections, i.e. conjugates of simple reflections, $t_1, t_2, \dots, t_m$ such that $w' = w \cdot t_1 \cdot t_2 \cdots t_m$ and $\ell(w') > \ell(w)$. When $m=1$, such a relation will be called \emph{elementary}. By \cite[Remark 6.4]{ML1}, the mutations correspond to elementary Bruhat relations, meaning that there exists a mutation $\mu \colon \J' \longrightarrow \J$ if and only if $f_{\J'} \circ \text{id}_{-n} \ge f_\J \circ \text{id}_{-n}$ and they differ by a reflection. So the $T$-action encodes information about the closure inclusion order.

\smallskip

To any $(k,N)$-juggling pattern $\J$ is associated a successor-closed subquiver of $Q(U_{[N]},SB)$ whose vertex set is $\{e^{(i)}_j \, \vert \, i \in \Z_N, j \in J_i\}$. It has exactly $k$ vertices for each $i \in \Z_N$, and vice versa, to any such subquiver one can associate a juggling pattern. Mutations then correspond to cutting the tail of a segment in the subquiver and gluing it at the tail end of a lower segment, i.e. removing the elements
\[x \in J'_a \, , \quad x+1 \in J'_{a+1} \, , \quad \dots \, , \quad x+l \in J'_{a+l}\]
and replacing them with
\[x+s \in J_a \, , \quad x+s+1 \in J_{a+1} \, , \quad \dots \, , \quad x+s+l \in J_{a+l}\]
where $s \ge 1$. Notice that $x+l+1$ is in both $J'_{a+l+1}$ and $J_{a+l+1}$, and if $x+s+l+1 \ne N+1$ then it too is in $J'_{a+l+1} \cap J_{a+l+1}$. For the same reason, $x+s-1$ and $x-1$ are in neither of $J'_{a-1}$ and $J_{a-1}$, and the latter might be 0.

\begin{example}\label{ex:mutation}
    The $(2,4)$-juggling patterns
 \[\J' = \begin{matrix}
        &\{1,2\}& \\
        \{1,4\}&&\{2,3\} \\
        &\{3,4\}&
    \end{matrix} \qquad \ge \qquad
    \J = \begin{matrix}
        &\{1,3\}& \\
        \{2,4\}&&\{2,4\} \\
        &\{3,4\}&
    \end{matrix}\]
correspond respectively to the following successor-closed subquivers (in red) of $Q(U_{[4]},SB)$
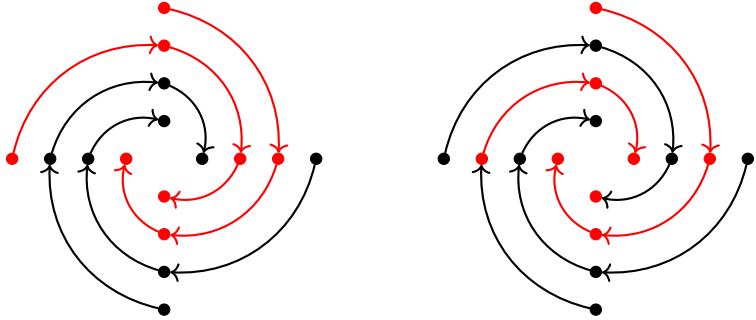
\begin{center}
 \begin{tikzpicture}
     \begin{scope}[every node/.style={circle, draw=black!100, fill=black!100, 
    very thin, inner sep = 0pt, minimum size = 1.5mm}]
    \node[fill=red!100, draw=red!100] (01) at (0,2) {};
    \node[fill=red!100, draw=red!100] (02) at (0,1.5) {};
    \node (03) at (0,1) {};
    \node (04) at (0,0.5) {};
    \node (11) at (2,0) {};
    \node[fill=red!100, draw=red!100] (12) at (1.5,0) {};
    \node[fill=red!100, draw=red!100] (13) at (1,0) {};
    \node (14) at (0.5,0) {};
    \node (21) at (0,-2) {};
    \node (22) at (0,-1.5) {};
    \node[fill=red!100, draw=red!100] (23) at (0,-1) {};
    \node[fill=red!100, draw=red!100] (24) at (0,-0.5) {};
    \node[fill=red!100, draw=red!100] (31) at (-2,0) {};
    \node (32) at (-1.5,0) {};
    \node (33) at (-1,0) {};
    \node[fill=red!100, draw=red!100] (34) at (-0.5,0) {};
\end{scope}

\begin{scope}[every edge/.style= 
              {draw=black,thick}]
\path [->] (01) edge[bend left=40,draw=red] (12);
\path [->] (02) edge[bend left=40,draw=red] (13);
\path [->] (03) edge[bend left=40] (14);
\path [->] (11) edge[bend left=40] (22);
\path [->] (12) edge[bend left=40,draw=red] (23);
\path [->] (13) edge[bend left=40,draw=red] (24);
\path [->] (21) edge[bend left=40] (32);
\path [->] (22) edge[bend left=40] (33);
\path [->] (23) edge[bend left=40,draw=red] (34);
\path [->] (31) edge[bend left=40,draw=red] (02);
\path [->] (32) edge[bend left=40] (03);
\path [->] (33) edge[bend left=40] (04);

\end{scope}
    \end{tikzpicture} \qquad \qquad
 \begin{tikzpicture}
     \begin{scope}[every node/.style={circle, draw=black!100, fill=black!100, 
    very thin, inner sep = 0pt, minimum size = 1.5mm}]
    \node[fill=red!100, draw=red!100] (01) at (0,2) {};
    \node (02) at (0,1.5) {};
    \node[fill=red!100, draw=red!100] (03) at (0,1) {};
    \node (04) at (0,0.5) {};
    \node (11) at (2,0) {};
    \node[fill=red!100, draw=red!100] (12) at (1.5,0) {};
    \node (13) at (1,0) {};
    \node[fill=red!100, draw=red!100] (14) at (0.5,0) {};
    \node (21) at (0,-2) {};
    \node (22) at (0,-1.5) {};
    \node[fill=red!100, draw=red!100] (23) at (0,-1) {};
    \node[fill=red!100, draw=red!100] (24) at (0,-0.5) {};
    \node (31) at (-2,0) {};
    \node[fill=red!100, draw=red!100] (32) at (-1.5,0) {};
    \node (33) at (-1,0) {};
    \node[fill=red!100, draw=red!100] (34) at (-0.5,0) {};
\end{scope}

\begin{scope}[every edge/.style= 
              {draw=black,thick}]
\path [->] (01) edge[bend left=40,draw=red] (12);
\path [->] (02) edge[bend left=40] (13);
\path [->] (03) edge[bend left=40,draw=red] (14);
\path [->] (11) edge[bend left=40] (22);
\path [->] (12) edge[bend left=40,draw=red] (23);
\path [->] (13) edge[bend left=40] (24);
\path [->] (21) edge[bend left=40] (32);
\path [->] (22) edge[bend left=40] (33);
\path [->] (23) edge[bend left=40,draw=red] (34);
\path [->] (31) edge[bend left=40] (02);
\path [->] (32) edge[bend left=40,draw=red] (03);
\path [->] (33) edge[bend left=40] (04);

\end{scope}

    \end{tikzpicture}    
    \end{center}
and the mutation $\mu \colon \J' \longrightarrow \J$, in the picture below, is given by removing the pink vertices and adding the blue ones, that is, shifting the pink segment towards the center of the spiral.
\begin{center}
    \begin{tikzpicture}
     \begin{scope}[every node/.style={circle, draw=black!100, fill=black!100, 
    very thin, inner sep = 0pt, minimum size = 1.5mm}]
    \node (01) at (0,2) {};
    \node[fill=pink!100, draw=pink!100] (02) at (0,1.5) {};
    \node[fill=blue!100, draw=blue!100] (03) at (0,1) {};
    \node (04) at (0,0.5) {};
    \node (11) at (2,0) {};
    \node (12) at (1.5,0) {};
    \node[fill=pink!100, draw=pink!100] (13) at (1,0) {};
    \node[fill=blue!100, draw=blue!100] (14) at (0.5,0) {};
    \node (21) at (0,-2) {};
    \node (22) at (0,-1.5) {};
    \node (23) at (0,-1) {};
    \node (24) at (0,-0.5) {};
    \node[fill=pink!100, draw=pink!100] (31) at (-2,0) {};
    \node[fill=blue!100, draw=blue!100] (32) at (-1.5,0) {};
    \node (33) at (-1,0) {};
    \node (34) at (-0.5,0) {};
\end{scope}

\begin{scope}[every edge/.style= 
              {draw=black,thick}]
\path [->] (01) edge[bend left=40] (12);
\path [->] (02) edge[bend left=40,draw=pink] (13);
\path [->] (03) edge[bend left=40,draw=blue] (14);
\path [->] (11) edge[bend left=40] (22);
\path [->] (12) edge[bend left=40] (23);
\path [->] (13) edge[bend left=40] (24);
\path [->] (21) edge[bend left=40] (32);
\path [->] (22) edge[bend left=40] (33);
\path [->] (23) edge[bend left=40] (34);
\path [->] (31) edge[bend left=40, draw=pink] (02);
\path [->] (32) edge[bend left=40,draw=blue] (03);
\path [->] (33) edge[bend left=40] (04);
\path [->] (31) edge[dashed,thin] (32);
\path [->] (02) edge[dashed,thin] (03);
\path [->] (13) edge[dashed,thin] (14);

\end{scope}
    \end{tikzpicture}
    \end{center}
\end{example}

\noindent For more details, see \cite[Section 6]{ML1}.

\section{Symplectic torus action} \label{sec:torusaction2}

The goal of this section is to define \emph{symplectic} mutations, and to find an algebraic torus with an action on $X(n,2n)^{sp}$ whose 1-dimensional orbits encode them. We start with the following statement, which is a reformulation of \cite[Lemma 4.4]{ours}:
\begin{lemma}\label{lem:correctionlemma}
    Let $\mu \colon \J'' \longrightarrow \J'$ be a mutation between $(k,2n)$-juggling patterns, and suppose that $\J''$ is symplectic while $\J'$ is not. Then there exist a symplectic $(k,2n)$-juggling pattern $\J \le \J'$ and a mutation $R\mu \colon \J' \longrightarrow \J$, that replaces precisely the elements of the sets of $\J'$ which are paired to the elements introduced by $\mu$. Analogously, given $\nu \colon \I' \longrightarrow \I$ with $\I$ symplectic and $\I'$ non-symplectic, there exist $\I''$ symplectic and $R\nu \colon \I'' \longrightarrow \I'$, that introduces precisely the elements of the sets of $\J$ which are paired to the elements removed by $\nu$.
\end{lemma}

\begin{proof}
	Suppose $\mu$ removes
	\begin{equation}\label{eq:oldsegment}
		x \in J''_a \, , \quad x+1 \in J''_{a+1} \, , \quad \cdots \, , \quad x+l \in J''_{a+l}
	\end{equation}
	from $\J''$ and replaces them with
	\begin{equation}\label{eq:newsegment}
		x+s \in J'_a \, , \quad x+s+1 \in J'_{a+1} \, , \quad \cdots \, , \quad x+s+l \in J'_{a+l}
	\end{equation}
	in $\J$. Now, $x+s+l+1$ is either in $J''_{a+l+1} \cap J'_{a+l+1}$ or is equal to $2n$. Also $x+l+1 \in J''_{a+l+1} \cap J'_{a+l+1}$. By assumption $\J'$ is not symplectic, and since the only difference from the symplectic $\J''$ is this segment, $\J'$ must contain the elements
	\[\tilde{x}-s-l \in J'_{-a-l} \, , \quad \tilde{x}-s-l+1 \in J'_{-a-l+1} \, , \quad \cdots \, , \quad \tilde{x}-s \in J'_{-a}\]
	but not the elements
	\[\tilde{x}-l \notin J'_{-a-l} \, , \quad \tilde{x}-l+1 \notin J'_{-a-l+1} \, , \quad \cdots \, , \quad \tilde{x} \notin J'_{-a} \, .\]
	If it did, they would be part of the added segment because they are not in $\J''$; but a mutation that removes some $y$ on a vertex and adds $\tilde{y}$ on the opposite vertex cannot connect a symplectic juggling pattern to a non-symplectic one. Furthermore, if $\tilde{x} < 2n$ then $\tilde{x}+1 \in J'_{-a+1}$. This holds because $x-1 \notin J''_{a-1}$ implies $\tilde{x}+1 \in J''_{-a+1}$, which cannot be part of the segment removed by $\mu$: if this was the case, $\tilde{x}+1$ would have to be of the form $x+m$ with $-a+1 \equiv_{2n} a+m$, which cannot happen because $m$ would have to be both odd and even. Then there is a mutation $R\mu \colon \J' \longrightarrow \J$ that removes precisely the first segment and adds the second one, since also $\tilde{x}-s-l-1 \notin J'_{-a-l-1}$ for parity reasons. We observe that $\J$ is symplectic since, by cardinality reasons, $R\mu$ removes all and only the elements paired by the symplectic form to those introduced by $\mu$, which made $\J'$ non-symplectic.
	
	The second half of the statement is proved analogously: if
	\[x \in I'_a \, , \quad x+1 \in I'_{a+1} \, , \quad \cdots \, , \quad x+l \in I'_{a+l}
	\]
	are replaced with
	\[x+s \in I_a \, , \quad x+s+1 \in I_{a+1} \, , \quad \cdots \, , \quad x+s+l \in I_{a+l}
	\]
	by $\nu$, then $\I'$ does not contain
	\[\tilde{x}-s-l \notin I_{-a-l} \, , \quad \tilde{x}-s-l+1 \notin I_{-a-l+1} \, , \quad \cdots \, , \quad \tilde{x}-s \notin I_{-a}\]
	because, since they are not in the removed segment, $\I$ would too. Instead, $\I'$ contains
	\[\tilde{x}-l \in I_{-a-l} \, , \quad \tilde{x}-l+1 \in I_{-a-l+1} \, , \quad \cdots \, , \quad \tilde{x} \in I_{-a} \, .\]	
	Since $x+s-1 \notin I_{a-1}$, $\tilde{x}-s+1$ is in $I_{-a+1}$ and therefore in $I'_{-a+1}$, since it cannot be one of the added elements. We conclude that there exists $\I''$ symplectic and a mutation $R\nu \colon \I'' \longrightarrow \I'$ which adds the appropriate elements. 
\end{proof}

\begin{definition}
    A pair of mutations as described in Lemma \ref{lem:correctionlemma} are called \emph{corrections} of one another.
\end{definition}

\begin{example}\label{ex:symplecticmutation}
Consider the juggling patterns
\[\J''= \begin{matrix}
        &\{2,4\}& \\
        \{1,4\}&&\{2,3\} \\
        &\{3,4\}&
    \end{matrix} \quad \ge \quad \J' = \begin{matrix}
        &\{2,4\}& \\
        \{3,4\}&&\{2,3\} \\
        &\{3,4\}&
    \end{matrix} \, \, .\]
They are linked by the following mutation $\mu$
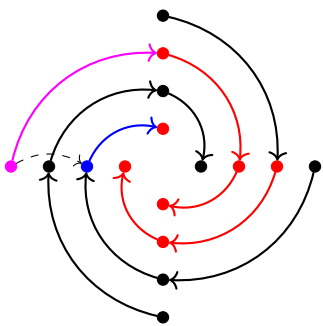
    \begin{center}
 \begin{tikzpicture}
     \begin{scope}[every node/.style={circle, draw=black!100, fill=black!100, 
    very thin, inner sep = 0pt, minimum size = 1.5mm}]
    \node (01) at (0,2) {};
    \node[fill=red!100, draw=red!100] (02) at (0,1.5) {};
    \node (03) at (0,1) {};
    \node[fill=red!100, draw=red!100] (04) at (0,0.5) {};
    \node (11) at (2,0) {};
    \node[fill=red!100, draw=red!100] (12) at (1.5,0) {};
    \node[fill=red!100, draw=red!100] (13) at (1,0) {};
    \node (14) at (0.5,0) {};
    \node (21) at (0,-2) {};
    \node (22) at (0,-1.5) {};
    \node[fill=red!100, draw=red!100] (23) at (0,-1) {};
    \node[fill=red!100, draw=red!100] (24) at (0,-0.5) {};
    \node[fill=pink!100, draw=pink!100] (31) at (-2,0) {};
    \node (32) at (-1.5,0) {};
    \node[fill=blue!100, draw=blue!100] (33) at (-1,0) {};
    \node[fill=red!100, draw=red!100] (34) at (-0.5,0) {};
\end{scope}

\begin{scope}[every edge/.style= 
              {draw=black,thick}]
\path [->] (01) edge[bend left=40] (12);
\path [->] (02) edge[bend left=40,draw=red] (13);
\path [->] (03) edge[bend left=40] (14);
\path [->] (11) edge[bend left=40] (22);
\path [->] (12) edge[bend left=40,draw=red] (23);
\path [->] (13) edge[bend left=40,draw=red] (24);
\path [->] (21) edge[bend left=40] (32);
\path [->] (22) edge[bend left=40] (33);
\path [->] (23) edge[bend left=40,draw=red] (34);
\path [->] (31) edge[bend left=40,draw=pink] (02);
\path [->] (32) edge[bend left=40] (03);
\path [->] (33) edge[bend left=40,draw=blue] (04);
\path [->] (31) edge[bend left=30,dashed,thin] (33);

\end{scope}
    \end{tikzpicture}
    \end{center}
which replaces the 1 in $J''_3$ with the 3 in $J'_3$. Notice that this makes it so that $\J''$ is symplectic while $\J'$ is not, since the $3$ in $J'_3$ is paired with the 2 in $J'_1$. The correction removes the 2 in $J'_1$ and replaces it with a 4, the number paired to the 1 removed by $\mu$. Therefore $R\mu$ is
    
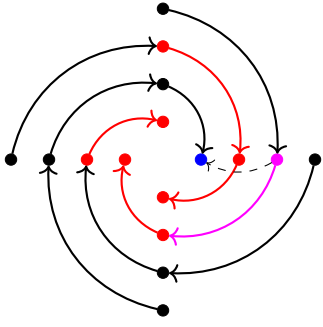
\begin{center}
 \begin{tikzpicture}
     \begin{scope}[every node/.style={circle, draw=black!100, fill=black!100, 
    very thin, inner sep = 0pt, minimum size = 1.5mm}]
    \node (01) at (0,2) {};
    \node[fill=red!100, draw=red!100] (02) at (0,1.5) {};
    \node (03) at (0,1) {};
    \node[fill=red!100, draw=red!100] (04) at (0,0.5) {};
    \node (11) at (2,0) {};
    \node[fill=pink!100, draw=pink!100] (12) at (1.5,0) {};
    \node[fill=red!100, draw=red!100] (13) at (1,0) {};
    \node[fill=blue!100, draw=blue!100] (14) at (0.5,0) {};
    \node (21) at (0,-2) {};
    \node (22) at (0,-1.5) {};
    \node[fill=red!100, draw=red!100] (23) at (0,-1) {};
    \node[fill=red!100, draw=red!100] (24) at (0,-0.5) {};
    \node (31) at (-2,0) {};
    \node (32) at (-1.5,0) {};
    \node[fill=red!100, draw=red!100] (33) at (-1,0) {};
    \node[fill=red!100, draw=red!100] (34) at (-0.5,0) {};
\end{scope}

\begin{scope}[every edge/.style= 
              {draw=black,thick}]
\path [->] (01) edge[bend left=40] (12);
\path [->] (02) edge[bend left=40,draw=red] (13);
\path [->] (03) edge[bend left=40] (14);
\path [->] (11) edge[bend left=40] (22);
\path [->] (12) edge[bend left=40,draw=pink] (23);
\path [->] (13) edge[bend left=40,draw=red] (24);
\path [->] (21) edge[bend left=40] (32);
\path [->] (22) edge[bend left=40] (33);
\path [->] (23) edge[bend left=40,draw=red] (34);
\path [->] (31) edge[bend left=40] (02);
\path [->] (32) edge[bend left=40] (03);
\path [->] (33) edge[bend left=40,draw=red] (04);
\path [->] (12) edge[bend left=30,dashed,thin] (14);

\end{scope}
    \end{tikzpicture}
    \end{center}
which produces
\[\J = \begin{matrix}
        &\{2,4\}& \\
        \{3,4\}&&\{3,4\} \\
        &\{3,4\}&
    \end{matrix} \le \J' \, .\]
\end{example}

\begin{remark}
    A mutation and its correction "commute", meaning that they do not replace the same elements and that it's possible to apply them in reverse order, which yields the same juggling pattern. This also implies that $RR\mu = \mu$ for any mutation $\mu$.
\end{remark}

\begin{definition}
    A \emph{symplectic mutation} $\J'' \overset{sp}{\longrightarrow} \J$ between two symplectic juggling patterns $\J'' \ge \J$ is either a single mutation $\J'' \longrightarrow \J$ which connects the two, or the composition of a pair of corrections $\mu \colon \J'' \longrightarrow \J'$ and $R\mu \colon \J' \longrightarrow \J$.
\end{definition}

\begin{remark}\label{rem:symplmutationsproperties}
    The mutation from Example \ref{ex:mutation} is symplectic, as well as the composition $\J'' \longrightarrow \J$ from Example \ref{ex:symplecticmutation}. Observe that if a single mutation $\J' \longrightarrow \J$ between symplectic juggling patterns changes some number $x$ to $y$ on a vertex $i$, it must also change $\tilde{y}$ to $\tilde{x}$ on $-i$. In particular, it affects an odd number of vertices, $\vert x-y \vert$ must be odd, and if $a=0$ or $n$ then $y=\tilde{x}$.
\end{remark}

To define the action we want, we need a weight function, different from the previous one, that induces a $\C^*$-action on $X(n,2n)$ which preserves $X(n,2n)^{sp}$. We choose
\[\widehat{\mathbf{wt}} \colon b_{j,p} \longmapsto p-\tilde{p} = -2n+2p-1 \, \, ,\]
so that the corresponding $\C^*$-action is
\begin{equation}\label{eq:newC*action}
z \cdot b_{j,p} = z^{-2n+2p-1} b_{j,p} = z^{p - \tilde{p}} b_{j,p} \, \, .
\end{equation}
Once again $p \in [2n]$ and $j \in [0, 2n-1]$. This new grading is attractive, increasing by 2 on each arrow, and the corresponding action preserves the symplectic form: since $b_{j,p}=e^{(j+p)}_p$, one has
\[(z \cdot v, z \cdot w)_\Omega = \sum_{p=1}^{2n} (-1)^{p+1} z^{p-\tilde{p}} \alpha_p \cdot z^{\tilde{p}-p} \beta_{\tilde{p}} (e^{(i)}_p, e^{(-i)}_{\tilde{p}})_\Omega = \sum_{p=1}^{2n} (-1)^{p+1} \alpha_p \cdot \beta_{\tilde{p}} (e^{(i)_p}, e^{(-i)}_{\tilde{p}})_\Omega = (v,w)_\Omega\]
for $v = \alpha_1 e^{(i)}_1 + \dots \alpha_{2n} e^{(i)}_{2n} \in U_{[2n]}^{(i)}$, $w = \beta_1 e^{(-i)}_1 + \dots \beta_{2n} e^{(-i)}_{2n} \in U_{[2n]}^{(-i)}$ and $z \in \C^*$. This implies that the action restricts to $X(n,2n)^{sp}$, therefore its attracting sets in $X(n,2n)^{sp}$ are the symplectic cells.

\smallskip

We now define a skeletal action of the algebraic torus $(\C^*)^{n+1}$ on $X(n,2n)^{sp}$. First, consider the action of $T \coloneqq (\C^*)^{2n+1}$ on $X(n,2n)$ given by
\begin{equation}\label{eq:newTaction}
    (z, \gamma_0, \dots, \gamma_{2n-1}) \cdot b_{j,p} = z^{-2n+2p-1} \cdot \gamma_j \cdot b_{j,p} = z^{p - \tilde{p}} \cdot \gamma_j \cdot b_{j,p} \, .\end{equation}
It is well defined, as it commutes with the map $s$, defined in \eqref{eq:s}, up to a factor of $z^2$: letting $t = (z, \gamma_0, \dots, \gamma_{2n-1}) \in T$, we have
\begin{gather*} s (t \cdot b_{j,p}) = s (z^{-2n+2p-1} \gamma_j b_{j,p}) = z^{-2n+2p-1} \gamma_j b_{j,p+1} = \\ = z^{-2n+2(p+1)-1} z^{-2} \gamma_j b_{j,p+1} = z^{-2} t \cdot (s(b_{j,p}))\, \, .\end{gather*}
Let $T^{sp}$ be the maximal subgroup of $T$ that preserves the symplectic form, and let $i \in \Z_N$ be a vertex and $t \in T$ as above. For any two vectors $v = \alpha_1 e^{(i)}_1 + \dots \alpha_{2n} e^{(i)}_{2n} \in U_{[2n]}^{(i)}$, $w = \beta_1 e^{(-i)}_1 + \dots \beta_{2n} e^{(-i)}_{2n} \in U_{[2n]}^{(-i)}$, we have
\begin{gather*}
    (v,w)_\Omega = \sum_{p=1}^{2n} (-1)^{p+1} \alpha_p \cdot \beta_{\tilde{p}} \, ;\\
    t \cdot v = \sum_{p=1}^{2n} \alpha_p \cdot z^{-2n+2p-1} \gamma_{i-p} \cdot e^{(i)}_p \, ;\\
    t \cdot w = \sum_{p=1}^{2n} \beta_{\tilde{p}} \cdot z^{2n-2p+1} \gamma_{-i+p-1} \cdot e^{(-i)}_{\tilde{p}} \, ;
\end{gather*}
so
\[(t \cdot v, t \cdot w)_\Omega = \sum_{p=1}^{2n} (-1)^{p+1} \alpha_p \cdot \beta_{\tilde{p}} \cdot z^0 \cdot \gamma_{i-p} \cdot \gamma_{-i+p-1} \, \, .\]
Therefore $t$ preserves the symplectic form if and only if $\gamma_{-j-1}=\gamma_j^{-1}$ for all $j \in [0,2n-1]$, where the subscript is taken modulo $2n$. Then $T^{sp}$ is an algebraic torus of dimension $n+1$, since the following group homomorphism $(\C^*)^{n+1} \longrightarrow (\C^*)^{2n+1}$ is injective with image $T^{sp}$:
\[(z, \gamma_0, \dots, \gamma_{n-1}) \longmapsto (z, \gamma_0, \dots, \gamma_{n-1}, \frac{1}{\gamma_{n-1}}, \dots, \frac{1}{\gamma_0}) \, .\]
The subvariety $X(n,2n)^{sp}$ is not closed under the action of $(\C^*)^{2n+1}$, but it is for the action of this smaller torus $T^{sp}$.

\begin{theorem}\label{thm:mainresult1}
The $T^{sp}$-action makes $X(n,2n)^{sp}$ into a GKM variety.
\end{theorem}

\begin{proof}
The fixed points for the $\C^*$-action on $X(n,2n)$ given by the grading $\widehat{\textbf{wt}}(b_{j,p}) = p-\tilde{p}$ are the juggling pattern points $p_\J$, the same as the fixed points for the $T$-action defined in \eqref{eq:newTaction}. This holds because the eigenvalues on the basis vectors of each vertex are all different \cite[Theorem 1]{giovanni}. Furthermore, such a $\C^*$-action is given by composing the cocharacter
\[\chi \colon z \longmapsto (z, 1, \dots, 1) \in T^{sp}\]
with the natural inclusion of $T^{sp}$ into $T$ and then acting via $T$. This implies that the $T^{sp}$-fixed points in $X(n,2n)$ are once again the juggling pattern points: therefore the $T^{sp}$-fixed points in $X(n,2n)^{sp}$ are the points $p_\J$ with $\J$ symplectic.

\smallskip

Now we study the 1-dimensional $T^{sp}$-orbits. The cells $C_\J^{sp}$ are $T^{sp}$-stable because any element
\[t = (z, \gamma_0, \dots, \gamma_{n-1}, \frac{1}{\gamma_{n-1}}, \dots, \frac{1}{\gamma_0})\]
of $T^{sp}$, seen as a collection of diagonal matrices in the standard basis on each vertex, factors as $\chi(z) \cdot g$, where $g$ is an element of $G^{sp}$, and because the action of $\chi(z) \in T^{sp}$ is exactly the action of $z \in \C^*$, under which the symplectic cells $C_\J^{sp}$ are stable. The element $t$ acts on a point of $X(n,2n)^{sp}$ with coordinates $(u^{(a)}_{i,j})$ as follows:
\[t \cdot \left( u^{(a)}_{i,j} \right) = \left( u^{(a)}_{i,j} \cdot z^{-2n+2p^{(a)}_i-1} \cdot z^{2n-2p^{(a)}_j+1} \cdot \gamma_{s^{(a)}_i} \cdot \gamma^{-1}_{s^{(a)}_j} \right) = \left( u^{(a)}_{i,j} \cdot z^{2p^{(a)}_i-2p^{(a)}_j} \cdot \gamma_{s^{(a)}_i} \cdot \gamma^{-1}_{s^{(a)}_j} \right) \, .\]
Here $s^{(a)}_i$ denotes the segment of the coefficient quiver that $e^{(a)}_i$ lies on, and it ranges in $[0,2n-1]$, while $p^{(a)}_i \in [2n]$ is its position in the segment. Since $t \in T^{sp}$, we have $\gamma_j^{-1} = \gamma_{-j-1}$. We know that a point has all zero coordinates if and only if it is fixed, so for the orbit to be one-dimensonal, at least one coordinate must be nonzero. Consider $a$, $i$ and $j$ be such that $u^{(a)}_{i,j} \ne 0$. Since the group acts component-wise via multiplication, for the orbit of a point to be 1-dimensional, any other nonzero coordinate $u^{(b)}_{l,m}$ must be such that $2p^{(a)}_i-2p^{(a)}_j = 2p^{(b)}_l-2p^{(b)}_m$ and that either $s^{(a)}_i = s^{(b)}_l$ and $s^{(a)}_j = s^{(b)}_m$ or $s^{(a)}_i = -s^{(b)}_m-1$ and $s^{(a)}_j = -s^{(b)}_l -1$. Notice that $s^{(b)}_l$ can never be equal to $s^{(b)}_m$ since a segment encounters each vertex only once, therefore the factors $\gamma_{s^{(b)}_l} \cdot \gamma^{-1}_{s^{(b)}_m}$ and $\gamma_{s^{(a)}_i} \cdot \gamma^{-1}_{s^{(a)}_j}$ cannot vanish.
%and they are equal either when $s^{(a)}_i = s^{(b)}_l$ and $s^{(a)}_j = s^{(a)}_m$ or when $s^{(a)}_i = -s^{(b)}_m-1$ and $s^{(a)}_j = -s^{(a)}_l -1$.
We deduce that the 1-dimensional orbits are finitely many in number. We end the proof by observing that, since the symplectic cells are affine spaces as a consequence of Lemma \ref{lem:tauforcells}, the rational cohomology of $X(n,2n)^{sp}$ vanishes in odd degree.

\end{proof}

\begin{remark}
Any 1-dimensional $T^{sp}$-orbit in $X(n,2n)^{sp}$ is either a 1-dimensional $T$-orbit, or the intersection of $X(n,2n)^{sp}$ with a 2-dimensional $T$-orbit in $X(n,2n)$. Furthermore, the $T^{sp}$-action is again locally linearizable, just like the $T$-action. This holds since $X(n,2n)^{sp}$ is a GKM variety \cite[1.2]{GKM}. Therefore any 1-dimensional $T^{sp}$-orbit in $X(n,2n)^{sp}$ contains two fixed points in its closure, and it must be contained in some cell $C_\J^{sp}$, since the cells are $T^{sp}$-stable. On the level of mutations, a 1-dimensional $T^{sp}$-orbit whose closure points are $p_{\J''}$ and $p_\J$, with $\J'' \ge \J$, translates to either a single mutation $\mu \colon \J'' \longrightarrow \J$ or the composition of a mutation $\nu \colon \J'' \longrightarrow \J'$ and its correction $R\nu \colon \J' \longrightarrow \J$, i.e. a symplectic mutation $\J'' \overset{sp}{\longrightarrow} \J$ \cite[Corollary 4.5]{ours}.
\end{remark}

\begin{prop}\label{prop:symplmutationsareorbits}
    Two symplectic $(n,2n)$-juggling patterns $\J' \ge \J$ are connected by a symplectic mutation if and only if $p_{\J'}$ and $p_\J$ are the closure points of a 1-dimensional $T^{sp}$-orbit.
\end{prop}

\begin{proof}
    The path
    \[V(t)_a \coloneqq \Span \Bigl ( \{e_j^{(a)} \, \vert \, j \in J_a \cap J'_a\} \cup \{e_{j-s}^{(a)} + t \cdot e_j^{(a)}  \, \vert \, j \in J_a \text{\textbackslash} J'_a \} \Bigr )\]
    connects $p_\J$ and $p_{\J'}$ \cite[Corollary 4.5]{ours} and is a 1-dimensional $T^{sp}$-orbit, because of the description of orbits in the proof of the previous theorem.
\end{proof}

\begin{prop}
    The $G^{sp}$-orbits in $X(n,2n)^{sp}$ coincide with the attracting loci of the $\C^*$-action given by the grading $\widehat{\emph{\textbf{wt}}}$.
\end{prop}

\begin{proof}
    We know that the attracting loci in $X(n,2n)$ for the $\C^*$-action given by the weight \textbf{wt}$(b_{j,p}) = p$ are the $G$-orbits. Also, notice from \eqref{eq:oldC*action} and \eqref{eq:newC*action} that the two $\C^*$-actions differ by the group morphism
    \[z \longmapsto z^2\]
    and by diagonal multiplication by the scalar $z^{-2n-1}$, which acts trivially on our quiver Grassmannian. Therefore the attracting loci for the two actions coincide. Moreover, the second action preserves $X(n,2n)^{sp}$, thus for any symplectic juggling pattern $\J$ we have
    \begin{gather*}
    \{V \in X(n,2n)^{sp} \, \vert \, \lim_{z \rightarrow 0} z \cdot V= p_\J\} = \{V \in X(n,2n) \, \vert \, \lim_{z \rightarrow 0} z \cdot V= p_\J\} \cap X(n,2n)^{sp} = \\ = C_\J \cap X(n,2n)^{sp} = C_\J^{sp} \, .\end{gather*}
\end{proof}

From the previous section, we know that the 1-dimensional orbits for the action of $T$ encode all the information necessary to recover the $G$-orbit closure inclusion order. We prove this to be true for $T^{sp}$ and the $G^{sp}$-orbits as well.

\section{Main results} \label{sec:mainresults}

\subsection{Closure inclusion order}

\begin{remark}
    For dimensional reasons, given a symplectic $(n,2n)$-juggling pattern $\J''$ and a mutation $\mu \colon \J'' \longrightarrow \J'$ as in Lemma \ref{lem:correctionlemma}, there is only one symplectic juggling pattern $\J$ obtained by mutating $\J'$ on the vertices opposite to those that $\mu$ changes. This mutation is $R\mu$.
\end{remark}

\begin{lemma}\label{lem:symplmutationsarereflections}
    Symplectic mutations in $X(n,2n)^{sp}$ correspond to elementary Bruhat relations in $\left(C^0_{n+1},S_C\right)$.
\end{lemma}

\begin{proof}
    Let $\J \le \J'$ be symplectic $(n,2n)$-juggling patterns such that their corresponding elements in $C^0_{n+1}$, $g_\J$ and $g_{\J'}$, differ by a reflection $t$ in type $C$. That is, either $t=(i, -i-1)$ for some $i$ or $t = (i,j)(-j-1,-i-1) \eqqcolon t_1 \circ t_2$ for some $i < j$, by Corollary \ref{cor:transpositionsintypeC}. In the first case $t$ is a reflection in type $A$ as well, so there is a simple mutation between $\J$ and $\J'$, which must be symplectic by definition. In the second case, we know that the juggling patterns corresponding to $g \le g \circ (i,j)$, for some $g \in A^0_{2n}$, differ only on the vertices $i-n+1, i-n+2, \dots, , j-n$ by Lemma \ref{lem:mutatedpermutations}. We can assume $t_1 \ne t_2$; then $t_1$ corresponds to a mutation on vertices $(i-n+1, \dots, j-n)$ and $t_2$ corresponds to a mutation on vertices $(-j-n, \dots, -i-n-1) = (-j+n, \dots, -i+n-1)$, which are opposite segments under the quiver symmetry. Since the mutation for $t_2$ produces a symplectic juggling pattern starting from a non-symplectic one (the one corresponding to $g_\J \circ t_1$) and changing only those vertices, it must be the correction of the mutation for $t_1$.

    \smallskip
    
    On the other hand, let now $\mu \colon \J'' \overset{sp}{\longrightarrow} \J$ be a symplectic mutation. If $\mu$ is a single mutation, then $f_{\J''} = g \circ \text{id}_n$ and $f_\J = g \circ (i,j) \circ \text{id}_n$ for some $g \in C^0_{n+1}$ and $(i,j) \in A^0_{2n}$ with $i<j$. Once again, the vertices that $\mu$ changes are those from $i-n+1$ to $j-n$ and they must satisfy $i-n+1 \equiv_{2n} -(j-n)$ by Remark \ref{rem:symplmutationsproperties}. Therefore the transposition $(i,j)$ is of the form $(i,-i-1)$. If instead $\mu$ is a composition $\J'' \overset{\nu}{\longrightarrow} \J' \overset{R\nu}{\longrightarrow} \J$, their respective bounded affine permutations must be of the form
    \[f_{\J''} = g \circ (i,j) \circ (i',j') \circ \text{id}_n \, , \quad f_{\J'}= g \circ (i,j) \circ \text{id}_n \, , \quad f_\J = g \circ \text{id}_n\]
    for some $g \in C^0_{n+1}$ and transpositions $(i,j), \, (i',j') \in A^0_{2n}$. Remember that $\nu$ and $R\nu$ must affect opposite vertices by Lemma \ref{lem:correctionlemma}. The first and last affected vertices by $\nu$ are $i-n+1$ and $j-n$, while those for $R\nu$ are $i'-n+1$ and $j'-n$, therefore we have $i-n+1 \equiv_{2n} j'+n$ and $j-n \equiv_{2n} -i'+n-1$. This implies that the transposition $(i',j')$ coincides with $(-j-1,-i-1)$, since $f_\J$, $f_{\J'}$ and $f_{\J''}$ are bounded.
\end{proof}

\begin{theorem}\label{thm:mainresult2}
    The closure of a symplectic cell in $X(n,2n)^{sp}$ is the intersection of $X(n,2n)^{sp}$ with the closure of the corresponding affine cell in $X(n,2n)$.
\end{theorem}

\begin{proof}
    One inclusion follows easily from Proposition \ref{prop:orbitintersection}, so we move on to the other one. Consider two symplectic juggling patterns $\J' \ge \J$, i.e. such that $p_\J \in \overline{C_{\J'}}$. By \cite{ML1}, there exist mutations $\mu_1, \dots, \mu_m$ that link them. These correspond to elementary Bruhat relations in $(A^0_{2n},S)$ given by reflections $t_1, \dots, t_m$ such that $g_\J \le g_\J \circ t_1 \le g_\J \circ t_1 \circ t_2 \le \dots \le g_{\J'}$. The Bruhat order in $(C^0_{n+1},S_C)$ coincides, via the natural monomorphism, with the Bruhat order in $(A^0_{2n},S)$ \cite[Proposition 3.1]{karpman}. So there exist elementary Bruhat relations in $(C^0_{n+1},S_C)$, given by some reflections $t_1^C, t_2^C, \dots, t_{m'}^C$ whose product is the same as the product of the $t_i$'s. By Proposition \ref{prop:symplmutationsareorbits}, the product by each of these translates, in $X(n,2n)^{sp}$, to 1-dimensional $T^{sp}$-orbits, which are contained in the $G^{sp}$-orbits. Therefore we have found a sequence of symplectic juggling patterns  $\J'' = \J^{(m')} \ge \J^{(m'-1)} \ge \dots \ge \J^{(1)} \ge \J^{(0)} = \J$ such that $p_{\J^{(i)}}$ is in the closure of the $G^{sp}$-orbit $C^{sp}_{\J^{(i+1)}}$. This concludes the proof.
\end{proof}

\begin{corollary}
    Any two $(n,2n)$-symplectic juggling patterns $\J'' \ge \J$ such that there is no $\J' \in JP(n,2n)^{sp}$ with $\J'' \ge \J' \ge \J$ are at most two mutations away.
\end{corollary}

\subsection{Dimension of the cells}

\begin{definition}
    For any element $g$ of the Coxeter group $(C^0_{n+1},S_C)$ we denote with $\ell^{sp}(g)$ its length with respect to $S_C$ and we call it \emph{symplectic length} of $g$. For $f = g \circ \text{id}_n \in \mathcal{B}_{(n,2n)}$ symplectic, we define $\ell^{sp}(f) \coloneqq \ell^{sp}(g)$.
\end{definition}

\begin{remark}\label{rem:pairsmodulo2n}
    Given two pairs $(x,y)$ and $(x',y')$ in $L(g)$, we have $x \equiv_{2n} x'$ if and only if $x=x'$ by definition. Moreover, if also $y \equiv_{2n} y'$ then $y=y'$.
\end{remark}

\begin{prop}\label{prop:Lchange}
    Let $g, s_i \in A^0_{2k}$ be an affine permutation and a simple reflection such that $\ell(g\circ s_i) = \ell(g)+1$, where we consider $i \in [0,2n-1]$. Then the pair $(i,i+1)$ is contained in $L(g\circ s_i)$ but not in $L(g)$. Moreover, the map
    \begin{align*}
        \iota_i \colon L(g) \longrightarrow & L(g\circ s_i) \\ 
        (x,y) \longmapsto & (s_i(x), s_i(y))
    \end{align*}
    is well defined up to equivalence modulo $2n$ of the components, is injective, and is such that the diagram
    \[\begin{tikzcd}
        L(g) \arrow[r, "\Phi"] \arrow[d, "\iota_i", hook] & L(Rg) \arrow[d, "\iota_{-i-2}", hook] \\
        L(g\circ s_i) \arrow[r, "\Phi"] & L(Rg\circ s_{-i-2})
    \end{tikzcd}\]
    is commutative.
\end{prop}

\begin{proof}
    Observe that $\ell(g\circ s_i) = \ell(g)+1$ is equivalent to $g(i) < g(i+1)$. Aside from checking the commutativity of the square, the statement is proved by \cite[Proposition 8.3.1]{BjornerBrenti}. We are left with a simple computation, the crucial point of which is that $R(s_i)=s_{-i-2}$. There is nothing to check if $x,y \notin {i, i+1}$, so suppose for example $x=i$ and that $y \le 2n-1$. By assumption, $y$ must differ from $i+1$. We have
    \begin{align*}
        \Phi(i,y) &= (2n-y-1, 2n-i-1) \\
        \iota_i (i,y) &= (i+1,y)
    \end{align*}
    therefore
    \begin{align*}
        \iota_{-i-2}\circ \Phi(i,y) &= (2n-y-1,2n-i-2) \\
        \Phi \circ \iota_i (i,y) &= (2n-y-1,2n-i-2) \, .
    \end{align*}
    The argument is similar in all the other nontrivial cases, listed below.
    \begin{itemize}
    	\item $x=i$ and $y \ge 2n$;
    	\item $x=i+1$ and $y \le 2n-1$;
    	\item $x=i+1$ and $y \ge 2n$;
    	\item $y=i$;
    	\item $i<2n-1$ and $y=i+1$;
    	\item $y=i+1=2n$.
    \end{itemize}
\end{proof}

\begin{remark}
    For a symplectic $(n,2n)$-bounded affine permutation $f$, the set $L(f)$ is equipped with an equivalence relation $\sim_\Phi$, where $(i,j) \sim_\Phi (i',j')$ if and only if $(i',j')$ is equal to either $(i,j)$ or $\Phi(i,j)$.
\end{remark}
%The next lemma is a known result, as is its corollary, but we still offer a proof.
\begin{lemma}\label{lem:simplength}
    The symplectic length of $f \in \mathcal{B}_{(n,2n)}$ equals the number of $\sim_\Phi$-equivalence classes in $L(f)$.
%    \[\grpfrac{L(f)}{\sim_\Phi} \, ,\]
\end{lemma}

\begin{proof}
    Let $g \in C^0_{n+1}$ be such that $f = g \circ \text{id}_n$. We can equivalently prove that $\ell^{sp}(g)$ is equal to the number of $\sim_\Phi$-equivalence classes in $L(g)$. We proceed by induction on $\ell^{sp}(g)$. If it is 0 there is nothing to prove. If instead $\ell^{sp}(g)=1$ then there are two cases:
    \begin{itemize}
        \item $g \in \{s_{-1}, s_{n-1}\}$, which implies $L(g)$ has one element, so one equivalence class;
        \item $g=s_is_{-i-2}$ for $i \in [0, n-2]$, so $L(g) = \{(i,i+1),(2n-i-2, 2n-i-1)\}$ and the two pairs make one equivalence class.
    \end{itemize}
    For the inductive step there is also more than one case. First suppose $\ell^{sp}(g\circ s_{-1})=\ell^{sp}(g)+1$. This implies that $\ell(g\circ s_{-1}) =\ell(g)+1$ and that $(2n-1,2n) \notin L(g)$, so by the previous proposition $L(g \circ s_{-1}) = \iota_{-1} L(g) \cup \{(2n-1,2n)\}$. Since $\Phi$ preserves this singleton, the number of classes in $L(g \circ s_{-1})$ is $\ell^{sp}(g)+1$ by induction. For $g \circ s_{n-1}$ things work analogously. Now let us fix $i \in [0,n-2]$ and suppose $\ell^{sp}(g \circ r_i) = \ell^{sp}(g)+1$. This means that $\ell(g \circ r_i)=\ell(g \circ s_i)+1=\ell(g \circ s_{-i})+1 = \ell(g)+2$. We apply Proposition \ref{prop:Lchange} twice and obtain
    \[\begin{tikzcd}
        L(g) \arrow[r, "\Phi"] \arrow[d, "\iota_{i}", hook] & L(g) \arrow[d, "\iota_{-i-2}", hook] \\
        L(g\circ s_i) \arrow[r, "\Phi"] \arrow[d, "\iota_{-i-2}", hook] & L(g\circ s_{-i-2}) \arrow[d, "\iota_i", hook]\\
        L(g\circ r_i) \arrow[r, "\Phi"] & L(g\circ r_{i})
    \end{tikzcd}\]
    where the two elements in $L(g \circ r_i)$ missing from the image of $L(g)$ are $(i,i+1)$ and $(2n-i-2, 2n-i-1)$, and they are in the same $\sim_\Phi$-equivalence class. Observe that since $i \in [0,n-2]$, $i, i+1, 2n-i-2$ and $2n-i-1$ are necessarily 4 different numbers.
\end{proof}

\begin{corollary}\label{cor:lengthplusfixed}
	For any symplectic $(n,2n)$-bounded affine permutation $f$, the number of pairs in $L(f)$ fixed by $\Phi$ is equal to
    \[2 \, \ell^{sp}(f) - \ell(f) \, .\]
\end{corollary}

\begin{lemma}\label{lem:howmanyfixed}
    Given $f \in \mathcal{B}_{(n,2n)}$ symplectic, $L(f)^\Phi$ is in bijection with
    \[\{ i \in [0, 2n-1] \, \, \vert \, \, f(i)-i > n \}\]
\end{lemma}
\begin{proof}
Fix a pair $(i,j) \in L(f)^\Phi$ and let $m_i = f(i)-i$, which satisfies $0 < m_i \le 2n$ since $f$ is bounded. Assume first that $j \le 2n-1$, so $(i,j)$ is fixed by $\Phi$ if and only if $j = 2n-i-1$ with $i \in [0,n-1]$, because $2n-i-1 = j > i$. Since $f$ is symplectic, we have $f(2n-i-1) - (2n-i-1) = 2n-m_i$ from \eqref{eq:symplecticpermutation}. It follows that
\[4n-i-m_i-1 = f(2n-i-1) < f(i) = i + m_i < 2n-i-1 +m_i\]
which implies $m_i>n$. If instead $j \ge 2n$, it must be equal to $4n-i-1$ and $i$ must be in $[n, 2n-1]$, because of \eqref{eq:ij-inequality}. This time we get
\[6n - i - m_i-1 = f(2n-i-1) +2n = f(4n-i-1) < f(i) = i + m_i < 4n-i-1+m_i\]
which implies once again $m_i > n$. The bijection from the statement is now obvious: $(i,j)$ in $L(f)$ is mapped to $i$, and in the opposite direction $i$ is mapped to $(i,2n-i-1)$ if $i \in [0,n-1]$ and to $(i,4n-i-1)$ if $i \in [n,2n-1]$.
\end{proof}

\begin{remark}\label{rem:maxfixed}
    Notice that this set has cardinality at most $n$, by \eqref{eq:symplecticpermutation}.
\end{remark}

\begin{example}
Let $\J$ be a maximal $(n,2n)$-juggling pattern, i.e. one for which the elements of $J_i$ are the representatives in $[2n]$ of the residue classes of $\{\bar{j}+i \, \, \vert \, \,  j \in J_0\}$. This means that $f_\J(i) = i+2n$ if and only if $2n \in J_i$ and $f_\J(i) = i$ otherwise. Let $J$ denote the set of representatives in $[0,2n-1]$ of these vertices, which has cardinality $n$. Thus computing the length of $f_\J$ is quite straightforward:
\[\ell(f_\J) = \big \vert \{(i,j) \in [0,2n-1] \times \Z \, \, \vert \, \, i \in J \, \land \, i < j < i+2n \, \, \land \, \, j \notin J \! \! \! \! \mod 2n\} \big \vert = n^2 = \dim \Gr(n,2n) \, .\]
Now suppose that $\J$ is symplectic. The pairs $(i,j)$ fixed by $\Phi$ are in bijection with $J$ because of Lemma \ref{lem:howmanyfixed}, so they are $n$ in number; we conclude that $\ell^{sp}(f_\J) = \frac{n(n+1)}{2} = \dim X(n,2n)^{sp}$, as stated in \cite[Proposition 4.12]{ours}. The maximal juggling patterns are the only ones with this symplectic length, by Corollary \ref{cor:lengthplusfixed} and Remark \ref{rem:maxfixed}.
\end{example}

\begin{theorem}\label{thm:mainresult3}
    For any symplectic $(n,2n)$-juggling pattern $\J$, the dimension of $C_\J^{sp}$ is equal to the length of the corresponding bounded affine permutation $f_\J$ in $C^0_{n+1}$.
\end{theorem}

\begin{proof}
    We start with a few observations. Let $\mu \colon \J' \longrightarrow \J$ be a mutation corresponding to a simple reflection in $A^0_{2n}$, and let $j \in J'_a$ be one of the removed elements, replaced with $i \in J_a$ (therefore $i>j$). Since $\dim C_{\J'} = \dim C_\J +1$ by \cite[Lemma 7.4]{ML1}, the affine coordinates for $C_{\J'}$ are one more than those for $C_\J$.  Remember that the affine coordinates for the cell $C_\J$ are parametrized by triples $(b,x,y)$ with $b \in \Z_{2n}$, $y \in J_b$, $x \notin J_b$, $x>y$ modulo the equivalence relation $(b,x,y)\sim (b+1,x+1,y+1)$. Therefore, the coordinates for both cells $C_{\J'}$ and $C_\J$ are parametrized by the same equivalence classes of triples, except $(a,i,j)$, because the mutation removes only that one segment and the conditions on the triples are preserved by the mutation. On the other hand, from Corollary \ref{cor:symplcoordinates} we know that the affine coordinates for $C^{sp}_\J$ are parametrized by the $\overset{sp}{\sim}$-equivalence classes of such triples, i.e. we also require $(b,x,y) \overset{sp}{\sim} (-b, \tilde{y}, \tilde{x})$.
    
    \smallskip 
    
    We proceed by induction on $\ell^{sp}(f_\J)$. The base case $f_\J = \text{id}_n$, i.e. when $\ell^{sp}(f_\J)=0$, is trivial. Now let $\J'' \overset{sp}{\longrightarrow} \J$ be a symplectic mutation corresponding to a simple reflection in $C^0_{n+1}$, i.e.
    \begin{equation}\label{eq:a}
    f_{\J''} = r_i \circ f_\J \, \land \, \ell^{sp}(f_{\J''}) = \ell^{sp}(f_\J)+1\end{equation}
    for some $i \in [-1,n-1]$, which implies $\ell(f_{\J''})=\ell(f_\J) +1$ if $i = -1, n-1$ and $\ell(f_{\J''})=\ell(f_\J) +2$ otherwise. We want to show that the claim holds for $\J''$, assuming that it holds for $\J$. In the first case, $\J'' \overset{sp}{\longrightarrow} \J$ is a single mutation $\mu \colon \J'' \longrightarrow \J$ which connects two symplectic juggling patterns. Thus, the segment which $\mu$ moves downwards must pass through either the vertex 0 or the vertex $n$, where it must replace some $j$ in $J''_a$, for $a \in \{0,n\}$, with $\tilde{j}>j$. In addition, if $\mu$ replaces some number $x$ with $x+s$ on a vertex $b$, then it must replace $\tilde{x}-s$ with $\tilde{x}$ on the vertex $-b$, and $s$ must be equal to $\tilde{j}-j$. Therefore the affine coordinates for $C_{\J''}^{sp}$ are parametrized by the same $\overset{sp}{\sim}$-classes of triples as those for $C_\J^{sp}$, plus the class of $(a,\tilde{j},j)$. In the second case, $\J'' \overset{sp}{\longrightarrow} \J$ is the composition of a pair of corrections $\nu \colon \J'' \longrightarrow \J'$ and $R\nu \colon \J' \longrightarrow \J$. If $\nu$ replaces some $y \in \J''_b$ with $x>y$, then $R\nu$ must replace $\tilde{x} \in \J'_{-b}$ with $\tilde{y}$, since $\J'$ is not symplectic. The $\sim$-classes of triples that parametrize the coordinates for $C_{\J''}$ are exactly those for $C_\J$ plus two more, which are are $(b,x,y)$ and $(-b, \tilde{y}, \tilde{x})$. Therefore, when taken modulo $\overset{sp}{\sim}$, $C_{\J''}^{sp}$ has one more equivalence class than $C_\J^{sp}$. To conclude, in both cases we have
    \[\dim C_{\J''}^{sp} = \dim C_\J^{sp}+1\]
    and by inductive hypothesis we have $\dim C_\J^{sp}= \ell^{sp}(f_\J)$, therefore it holds for $\J''$ as well.
\end{proof}

\begin{appendices}
\section{Equivariant cohomology}\label{appendix}
We want to compute the $T^{sp}$-equivariant cohomology ring for $X(2,4)^{sp}$. Let us recall how to compute $H_T^\bullet(X)$ for a GKM variety $X$ equipped with the action of a torus $T$. First, we define the \emph{moment graph} of the $T$-action on $X$. The underlying graph is the 1-skeleton of the $T$-action: its vertex set is the fixed-point set $X^T$, and for every 1-dimensional orbit $O$ there is an edge %$x_{\scaleto{O}{3pt}}\, \tikz{\draw[white] (0, -0.08) -- (0, 0.08); \draw (0,0) -- (0.5, 0);}\, y_{\scaleto{O}{3pt}}$ 
between its closure points $x_{\scaleto{O}{3pt}}$ and $y_{\scaleto{O}{3pt}}$, labeled by a character $\chi_{\scaleto{O}{3pt}} \colon T \rightarrow \C^*$. This character encodes information about the $T$-action on $O$, i.e. gives an isomorphism $O \cong \C \P^1$ that is $T$-equivariant. Equivalently, the edge can be labeled by the corresponding linear map $\alpha_{\scaleto{O}{3pt}} \colon \mathfrak{t} \longrightarrow \C$, where $\mathfrak{t}$ is the Lie algebra of $T$. The character on each edge is unique up to sign, and the sign is irrelevant for the purposes of computing equivariant cohomology. Then, by \cite[Theorem 1.2.2]{GKM}, the pullback map
\[H^\bullet_T(X) \longrightarrow H^\bullet_T(X^T) = \bigoplus_{x \in X^T}  H^\bullet_T(\text{pt}) \cong \bigoplus_{x \in X^T} S(\mathfrak{t}^*) \]
is injective with image
\[\bigl \{\left(f_x\right)_{x \in X^T} \, \, \vert \, \, f_{x_{\scaleto{O}{3pt}}} - f_{y_{\scaleto{O}{3pt}}} \in \alpha_{\scaleto{O}{3pt}} \! \cdot \! S(\mathfrak{t}^*) \, \, \text{for any 1-dimensional orbit} \, \, O \bigr\}.\]
Here $S(\mathfrak{t}^*)$ is the symmetric algebra over the dual vector space $\mathfrak{t}^*$, i.e. the algebra of polynomials over $\mathfrak{t}$.

\smallskip

Now let us go back to our case. When $n=2$ the torus $T^{sp}$ has dimension 3, and we denoted its generic element with $(z, \gamma_0, \gamma_1, \gamma_1^{-1}, \gamma_0^{-1})$ for $z, \gamma_0, \gamma_1 \in \C^*$. Therefore, the natural basis for $\mathfrak{t} = \{(a,b,c,-c,-b) \, \, \vert \, \, a,c,b \in \C\}$, is given by $(1,0,0,0,0)$, $(0,1,0,0,-1)$ and $(0,0,1,-1,0)$. The polynomial algebra we are considering is $\C[x,y_0,y_1]$, where $x, y_0,y_1$ is the basis dual to the one just above. The moment graph for this action on $X(2,4)^{sp}$, shown below, has the symplectic juggling patterns as vertices, which share an edge if and only if there is a symplectic mutation between them, by Proposition \ref{prop:symplmutationsareorbits}. Each juggling pattern $\J$ is written as $\begin{smallmatrix}
    & J_0 & \\
    J_3 & & J_1 \\
    & J_2 &
\end{smallmatrix}$, and the set $\{i,j\}$ is shortened to simply $ij$.
    
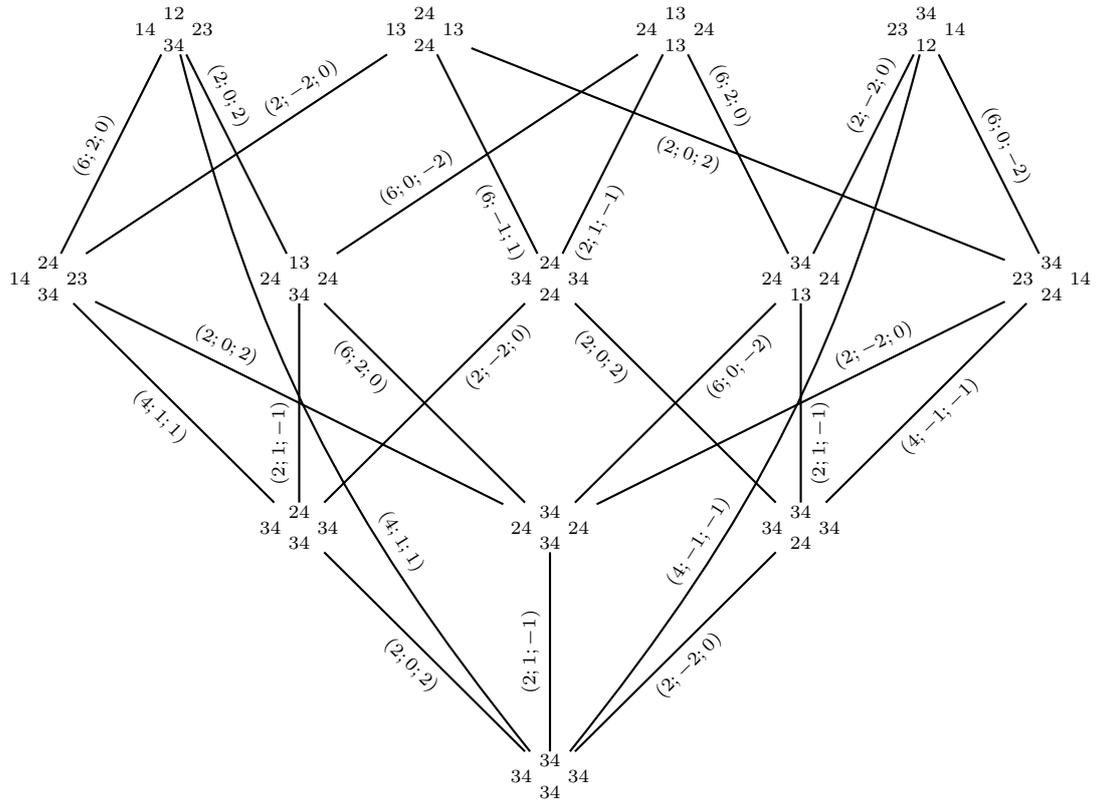
\begin{center}
    \begin{tikzpicture}
     \begin{scope}[every node/.style={rectangle, draw=black!0, fill=black!0,
     very thin,inner sep=2pt,minimum size=1mm%, font=\scriptsize
     }, scale=1.65]
    \node (1) at (0,0) {$\begin{smallmatrix}
        &34& \\
        34&&34 \\
        &34&
    \end{smallmatrix}$};
    \node (2) at (-2,2) {$\begin{smallmatrix}
        &24& \\
        34&&34 \\
        &34&
    \end{smallmatrix}$};
    \node (3) at (0,2) {$\begin{smallmatrix}
        &34& \\
        24&&24 \\
        &34&
    \end{smallmatrix}$};
    \node (4) at (2,2) {$\begin{smallmatrix}
        &34& \\
        34&&34 \\
        &24&
    \end{smallmatrix}$};
    \node (5) at (-4,4) {$\begin{smallmatrix}
        &24& \\
        14&&23 \\
        &34&
    \end{smallmatrix}$};
    \node (6) at (-2,4) {$\begin{smallmatrix}
        &13& \\
        24&&24 \\
        &34&
    \end{smallmatrix}$};
    \node (7) at (0,4) {$\begin{smallmatrix}
        &24& \\
        34&&34 \\
        &24&
    \end{smallmatrix}$};
    \node (8) at (2,4) {$\begin{smallmatrix}
        &34& \\
        24&&24 \\
        &13&
    \end{smallmatrix}$};
    \node (9) at (4,4) {$\begin{smallmatrix}
        &34& \\
        23&&14 \\
        &24&
    \end{smallmatrix}$};
    \node (10) at (-3,6) {$\begin{smallmatrix}
        &12& \\
        14&&23 \\
        &34&
    \end{smallmatrix}$};
    \node (11) at (-1,6) {$\begin{smallmatrix}
        &24& \\
        13&&13 \\
        &24&
    \end{smallmatrix}$};
    \node (12) at (1,6) {$\begin{smallmatrix}
        &13& \\
        24&&24 \\
        &13&
    \end{smallmatrix}$};
    \node (13) at (3,6) {$\begin{smallmatrix}
        &34& \\
        23&&14 \\
        &12&
    \end{smallmatrix}$};
\end{scope}

\begin{scope}[every edge/.style= 
              {draw=black,thick}, every edge quotes/.append style={font=\scriptsize}, sloped]
\path [-] (1) edge["$(2;0;2)$"'] (2);
\path [-] (1) edge["$(2;1;-1)$"] (3);
\path [-] (1) edge["$(2;-2;0)$"'] (4);
\path [-] (2) edge["$(4;1;1)$"'] (5);
\path [-] (2) edge[] node[above, pos=0.3]  {\scriptsize $(2;1;-1)$} (6);
\path [-] (2) edge[] node[below, pos=0.8]  {\scriptsize $(2;-2;0)$} (7);
\path [-] (3) edge[] node[above, pos=0.7]  {\scriptsize $(2;0;2)$} (5);
\path [-] (3) edge[] node[below, pos=0.75]  {\scriptsize $(6;2;0)$} (6);
\path [-] (3) edge[] node[below, pos=0.75]  {\scriptsize $(6;0;-2)$} (8);
\path [-] (3) edge[] node[above, pos=0.7]  {\scriptsize $(2;-2;0)$} (9);
\path [-] (4) edge[] node[below, pos=0.8]  {\scriptsize $(2;0;2)$} (7);
\path [-] (4) edge[] node[below, pos=0.3]  {\scriptsize $(2;1;-1)$} (8);
\path [-] (4) edge["$(4;-1;-1)$"'] (9);
\path [-] (5) edge["$(6;2;0)$"] (10);
\path [-] (5) edge[] node[above, pos=0.75]  {\scriptsize $(2;-2;0)$} (11);
\path [-] (6) edge[] node[above, pos=0.75]  {\scriptsize $(2;0;2)$} (10);
\path [-] (6) edge[] node[above, pos=0.3]  {\scriptsize $(6;0;-2)$} (12);
\path [-] (7) edge[] node[below, pos=0.2]  {\scriptsize $(6;-1;1)$} (11);
\path [-] (7) edge[] node[below, pos=0.2]  {\scriptsize $(2;1;-1)$} (12);
\path [-] (8) edge[] node[above, pos=0.75]  {\scriptsize $(6;2;0)$} (12);
\path [-] (8) edge[] node[above, pos=0.75]  {\scriptsize $(2;-2;0)$} (13);
\path [-] (9) edge[] node[below, pos=0.58]  {\scriptsize $(2;0;2)$} (11);
\path [-] (9) edge["$(6;0;-2)$"] (13);
\path [-] (1) edge[bend left = 12] node[above, pos=0.3]  {\scriptsize $(4;1;1)$} (10);
\path [-] (1) edge[bend right=12] node[above, pos=0.3]  {\scriptsize $(4;-1;-1)$} (13);
\end{scope}
    \end{tikzpicture}
    \end{center}
The edge label $(a,b,c)$ stands for the polynomial $ax+by_0+cy_1$. Let us denote the vertices with numbers 1-13 from the bottom up and from left to right. A basis $(\xi_i)_{i=1}^{13}$ for the equivariant cohomology is given by equivariant classes of the closures of the cells $C_\J^{sp}$, and is shown in the following tables. Each entry shows the component on vertex $j$ (which labels the column) of the basis element $\xi_i$. Each table shows the polynomials that make up the basis elements of a given degree, as stated in the top left corner. Since the basis vector relative to a vertex $\J$ has nonzero components only on vertices $\J' \ge \J$, we leave out the unnecessary columns. Note that the tables for positive degrees were split into two for spacing reasons.

\begin{table}[h!]
\scriptsize
\centering
\begin{tabular}{||c||c c c c c c c c c c c c c||} 
 \hline
 Deg. 0 & 1 & 2 & 3 & 4 & 5 & 6 & 7 & 8 & 9 & 10 & 11 & 12 & 13 \\
 \hline \hline
 $\xi_1$ & 1 & 1 & 1 & 1 & 1 & 1 & 1 & 1 & 1 & 1 & 1 & 1 & 1\\ 
 \hline
\end{tabular}
    \label{tab:basis0}
\end{table}

\begin{table}[h!]
\scriptsize
\centering
\begin{tabular}{||c||c c c c c c||} 
 \hline
 Deg. 1 & 2 & 3 & 4 & 5 & 6 & 7 \\
 \hline \hline
 $\xi_2$ & $2x+2y_1$ & 0 & 0 &  $2x+2y_1$ & $6x+2y_0$ & $2x+2y_1$\\
 \hline
 $\xi_3$ & 0 & $2x+y_0-y_1$ & 0 & $4x+y_0+y_1$ & $2x+y_0-y_1$ & 0 \\
 \hline
 $\xi_4$ & 0 & 0 & $2x-2y_0$ & 0 & 0 & $2x-2y_0$\\
 \hline
\end{tabular}
    \label{tab:basis1a}
\end{table}

\newpage
\newgeometry{left=1.25cm, right=1.25cm}

\begin{table}[h!]
\scriptsize
\centering
\begin{tabular}{||c||c c c c c c||} 
 \hline
 Deg. 1 & 8 & 9 & 10 & 11 & 12 & 13 \\
 \hline \hline
 $\xi_2$ & 0 & 0 & $8x+2y_0+2y_1$ & $2x+2y_1$ & $6x+2y_0$ & 0 \\
 \hline
 $\xi_3$ & $2x+y_0-y_1$ & $4x-y_0-y_1$ & $4x+y_0+y_1$ & 0 & $2x+y_0-y_1$ & $4x-y_0-y_1$ \\
 \hline
 $\xi_4$ & $6x-2y_1$ & $2x-2y_0$ & 0 & $2x-2y_0$ & $6x-2y_1$ & $8x-2y_0-2y_1$\\
 \hline
\end{tabular}
    \label{tab:basis1b}
\end{table}

\begin{table}[h!]
\scriptsize
\centering
\begin{tabular}{||c||c c c c c||} 
\hline
 Deg. 2 & 5 & 6 & 7 & 8 & 9 \\
 \hline \hline
 $\xi_5$ & $(4x+y_0+y_1)(2x+2y_1)$ & 0 & 0 & 0 & 0\\
 \hline
 $\xi_6$ & 0 & $(2x+y_0-y_1)(6x+2y_0)$ & 0 & 0 & 0\\ 
 \hline
 $\xi_7$ & 0 & 0 & $(2x-2y_0)(2x+2y_1)$ & 0 & 0 \\
 \hline
 $\xi_8$ & 0 & 0 & 0 & $(2x+y_0-y_1)(6x-2y_1)$ & 0 \\
 \hline
 $\xi_9$ & 0 & 0 & 0 & 0 & $(4x-y_0-y_1)(2x-2y_0)$ \\
 \hline
\end{tabular}
    \label{tab:basis2a}
\end{table}

\begin{table}[h!]
\scriptsize
\centering
\begin{tabular}{||c||c c c c||} 
\hline
 Deg. 2 & 10 & 11 & 12 & 13 \\
 \hline \hline
 $\xi_5$ & $(4x+y_0+y_1)(2x+2y_1)$ & $(6x-y_0+y_1)(2x+2y_1)$ & 0 & 0\\
 \hline
 $\xi_6$ & $(4x+y_0+y_1)(6x+2y_0)$ & 0 & $(2x+y_0-y_1)(6x+2y_0)$ & 0 \\ 
 \hline
 $\xi_7$ & 0 & $(2x-2y_0)(2x+2y_1)$ & $(6x-2y_1)(6x+2y_0)$ & 0 \\
 \hline
 $\xi_8$ & 0 & 0 & $(2x+y_0-y_1)(6x-2y_1)$ & $(4x-y_0-y_1)(6x-2y_1)$ \\
 \hline
 $\xi_9$ & 0 & $(6x-y_0+y_1)(2x-2y_0)$ & 0 & $(4x-y_0-y_1)(2x-2y_0)$ \\
 \hline
\end{tabular}
    \label{tab:basis2b}
\end{table}

\begin{table}[h!]
\scriptsize
\centering
\begin{tabular}{||c||c c||} 
 \hline
 Deg. 3 & 10 & 11 \\
 \hline \hline
 $\xi_{10}$ & $(6x+2y_0)(4x+y_0+y_1)(2x+2y_1)$ & 0\\ 
 \hline
 $\xi_{11}$ & 0 & $(2x-2y_0)(6x-y_0+y_1)(2x+2y_1)$ \\
 \hline
\end{tabular}
    \label{tab:basis3a}
\end{table}

\begin{table}[h!]
\scriptsize
\centering
\begin{tabular}{||c||c c||} 
 \hline
 Deg. 3 & 12 & 13 \\
 \hline \hline
 $\xi_{12}$ & $(6x-2y_1)(2x+y_0-y_1)(6x+2y_0)$ & 0 \\
 \hline
 $\xi_{13}$ & 0 & $(2x-2y_0)(4x-y_0-y_1)(6x-2y_1)$ \\
 \hline
\end{tabular}
    \label{tab:basis3b}
\end{table}
\restoregeometry
\end{appendices}

\newpage

%bibliografia
  \bibliographystyle{alpha}
  \bibliography{books}

\end{document}